\pdfoutput=1
\documentclass[11pt]{amsart}
\usepackage{amsmath}
\usepackage{amssymb}
\usepackage{mathtools}
\usepackage{enumerate}
\usepackage{slashed}
\usepackage{caption}
\usepackage{subcaption}
\usepackage{newlfont}
\usepackage{amsrefs}
\usepackage{comment}
\usepackage[normalem]{ulem}
\usepackage{accents}
\usepackage{leftidx}

\usepackage{amsthm,verbatim,amsfonts,amscd}
\usepackage{graphics}
\usepackage{tensor}
\usepackage{enumitem}
\usepackage[utf8]{inputenc}
\usepackage{hyperref}

\usepackage{harpoon}
\usepackage[pdftex]{graphicx}

\usepackage{esint}
\usepackage{relsize}
\usepackage{slashed}

\usepackage{tikz-cd}

\theoremstyle{plain}
\newtheorem{theorem}{Theorem}[section]
\newtheorem{lemma}[theorem]{Lemma}

\newtheorem{corollary}[theorem]{Corollary}

\newtheorem{proposition}[theorem]{Proposition}

\theoremstyle{definition}
\newtheorem{remark}[theorem]{Remark}

\newtheorem{definition}[theorem]{Definition}

\numberwithin{equation}{section}

\def\R{\mathbb{R}}

\def\sup{\operatorname{sup}}

\def\Tr{\operatorname{Tr}}

\def\sup{\operatorname{sup}}

\def\dom{\operatorname{dom}}

\theoremstyle{plain}

\numberwithin{equation}{section}

\allowdisplaybreaks

\begin{document}

\title[PMT for Creased Initial Data]{The Positive Mass Theorem for Creased Initial Data}

\author[Kazaras]{Demetre Kazaras}
\address{Department of Mathematics, Michigan State University, East Lansing, MI 48824, USA}
\email{kazarasd@msu.edu}

\author[Khuri]{Marcus Khuri}
\address{Department of Mathematics, Stony Brook University, Stony Brook, NY, 11794, USA}
\email{marcus.khuri@stonybrook.edu}

\author[Lin]{Michael Lin}
\address{Department of Mathematics, Duke University, Durham, NC, 27708, USA}
\email{michael.lin956@duke.edu}



\thanks{M. Khuri acknowledges the support of NSF Grants DMS-2104229 and DMS-2405045.}

\begin{abstract}
We establish a spacetime positive mass theorem and rigidity statement for asymptotically flat spin initial data sets with a codimension one singularity controlled by a matching Bartnik data condition involving spacetime rotations, and discuss applications. This generalizes several previous works on the topic, including results of Miao, Tsang, and Shi-Tam.
\end{abstract}

\maketitle

\section{Introduction} 
\label{sec1} \setcounter{equation}{0}
\setcounter{section}{1}


Given a connected oriented $n$-dimensional manifold with boundary $M^n$ of dimension $n\geq 3$, a Riemannian metric $g$, and a symmetric $2$-tensor $k$, we refer to the triple $(M^n,g,k)$ as an \emph{initial data set}. This terminology refers to the spacetime setting in which $(M^n,g)$ is a spacelike hypersurface with second fundamental form $k$ in an $(n+1)$-dimensional Lorentzian manifold, and the triple plays the role of initial position and velocity for the Einstein equations.  
These objects satisfy the constraint equations
\begin{equation}
\mu=\frac12 \left(R_g+(\mathrm{Tr}_g k)^2-|k|_g^2 \right),\quad\quad J= \mathrm{div}_g\left(k-(\mathrm{Tr}_g k)g\right),
\end{equation}
where $R_g$ is the scalar curvature and $\mu$, $J$ represent the matter energy and momentum densities. It will always be assumed that these latter quantities are integrable $\mu,J\in L^1(M^n)$. We will consider {\emph{asymptotically flat}} initial data sets. This condition means that there is a compact set $\mathcal{C}\subset M^n$ such that $M^n\setminus \mathcal{C}$ consists of a finite number of pairwise disjoint ends $\cup_{\ell=1}^{\ell_0} M_{\ell}^n$, each of which is diffeomorphic to the compliment of a ball in Euclidean space space $\mathbb{R}^n\setminus B$ and in the coordinates $x$ provided by this diffeomorphism
\begin{equation}
|\partial^{l}(g_{ij}-\delta_{ij})(x)|=O(|x|^{-q-l}),\quad l=0,1,2,\qquad
|\partial^{l}k_{ij}(x)|=O(|x|^{-q-1-l}),\quad l=0,1,
\end{equation}
for some $q>\frac{n-2}{2}$. In this setting the ADM energy $E$ and linear momentum $P=(P_1,\dots,P_n)$ of each end are well-defined \cites{Bartnik,chrumass} and given by
\begin{align}
E&=\lim_{r\to\infty} \frac{1}{2(n-1)\omega_{n-1}}\int_{S_r}\sum_{i=1}^n(g_{ij,i}-g_{ii,j})\upsilon^jdA,\\
P_i&=\lim_{r\to\infty} \frac{1}{(n-1)\omega_{n-1}}\int_{S_r}(k_{ij}-(\mathrm{Tr}_gk) g_{ij})\upsilon^jdA,
\end{align}
where $\omega_{n-1}$ denotes the volume of the unit $(n-1)$-dimensional sphere and $\upsilon$ is the unit outward normal to the coordinate sphere $S_r$ in the given end. The Lorentz length $m=\sqrt{E^2-|P|^2}$ is then referred to as the ADM mass of the end. The positive mass theorem (PMT) asserts that if $(M^n,g)$ is complete and the \textit{dominant energy condition} (DEC) $\mu\geq|J|_g$ holds, then $E\geq |P|$. Various versions of this result were established in dimensions $3\leq n\leq 7$ by Schoen-Yau \cite{SY1} and Eichmair \cite{Eichmair} using the Jang equation and by Eichmair-Huang-Lee-Schoen \cite{EHLS} using marginally outer trapped surfaces, it was proven in all dimensions for spin manifolds by Witten \cites{ParkerTaubes,Witten} and Bartnik \cite{Bartnik}, and when $n=3$ spacetime harmonic functions were employed by Hirsch-Kazaras-Khuri \cite{HKK}. Moreover, an alternative approach has been put forth by Lohkamp \cite{Lohkamp} to treat the general case without dimensional or spin restrictions.

In this work we will study the spacetime positive mass theorem for initial data sets with singularities. Let us recall previous results in this direction beginning with the time symmetric case, when $k=0$. One of the earliest such works is due to Shi-Tam \cite{ShiTam}, in the context of the their proof of positivity for the Brown-York quasi-local mass. They consider metrics of regularity $g\in C^0(M^n)\cap C^\infty(M^n\setminus\Sigma^{n-1})$ where $\Sigma^{n-1}\subset M^n$ is a closed hypersurface bounding a precompact region $M^n_-$. They show that in the spin case the PMT is valid if the mean curvatures of $\partial M^n_{\pm}$ agree $H_+ =H_-$, where $M^n_+=M^n \setminus \overline{M}^n_-$ and the unit normal for both sides points out of $M^n_-$. Miao \cite{Miao} generalized this beyond the spin case with a local smoothing argument which only requires $H_- \geq H_+$, and McFeron-Sz\'{e}kelyhidi \cite{McFeronSzekelyhidi} later obtained the same result utilizing a different smoothing technique via Ricci flow. In 3-dimensions Hirsch-Miao-Tsang \cite{HMT} produce a mass formula in this context using harmonic functions. When $n<8$ or $M^n$ is spin, Lee \cite{LeeLip} has established a PMT for metrics of regularity $C^2(M^n\setminus \mathcal{S})\cap \mathrm{Lip}(M^n)$ where the singular set $\mathcal{S}$ has Minkowski dimension less than $n/2$. Using spinorial techniques, Lee-LeFloch \cite{LeeLeFloch} obtain a version of the theorem for Sobolev regularity $C^0(M^n)\cap W^{1,n}_{loc}(M^n)$. In the nonspin case Yuqiao \cite{Yuqiao} imposed extra curvature conditions to find a similar result. Moreover, Grant-Tassotti \cite{GrantTassotti} employed a modification of the Miao smoothing technique to treat the case of metrics with $W^{2,n/2}_{loc}$-regularity. Working in a weaker setting, Li-Mantoulidis \cite{LiMantoulidis} show mass positivity for $g\in L^\infty(M^n)\cap C^2(M^n\setminus \mathcal{S})$ where $\mathcal{S}$ is a codimension 2 submanifold satisfying a type of acute cone-angle condition, while Shi-Tam \cite{ShiTam18} obtain the result for singular sets of codimension at least 2 for metrics having $W^{1,p}_{loc}(M^n)$-regularity where $p>n$. In the case $g\in L^\infty(M^n)\cap C^2(M^n\setminus \mathcal{S})$ for a closed submanifold $\mathcal{S}$ of codimension at least 3, the positive mass theorem is known in dimensions 3 \cite{LiMantoulidis} and 4 \cite{Kaz}. 

Consider now the nontime-symmetric case, where much less is known. Utilizing spacetime harmonic functions in the 3-dimensional setting, Tsang \cite{Tsang2022} generalized the results of Shi-Tam \cite{ShiTam} and Miao \cite{Miao} by observing that the PMT holds if $H_- -H_+ \geq |\pi_- (\cdot,\nu) -\pi_+ (\cdot,\nu)|_{g}$ on $\Sigma^2$, where $\pi_{\pm}=k_{\pm} -(\mathrm{Tr}_{g_{\pm}}k_{\pm})g_{\pm}$ is the momentum tensor associated with $M_{\pm}^3$ and $\nu$ is the unit normal pointing out of $M^3_-$.
In \cite{Shibuya}, Shibuya extended the Lee-LeFloch theorem \cite{LeeLeFloch} to the spacetime setting under the assumption that the dominant energy conditions holds in the distributional sense. Furthermore, it should be mentioned that there are rigidity statements for many of the results discussed above, some of which will be mentioned below. To describe the results of the current paper, we begin with two definitions.

\begin{definition}
{\emph{Bartnik data}} is a closed orientable Riemannian $(n-1)$-manifold $(\Sigma^{n-1},\gamma)$ with a triple $(\mathcal{N},\vec{H},\beta)$ satisfying the following properties. 
\begin{enumerate}[label=(\roman*), topsep=3pt, itemsep=0ex]
\item $\mathcal{N}\to\Sigma^{n-1}$ is a trivial rank two vector bundle with structure group $SO^+(1,1)$, equipped with an invariant bundle metric $\langle\cdot,\cdot\rangle$.
\item $\vec{H}$ is a section of $\mathcal{N}$.
\item $\beta$ is a $1$-form on $\Sigma^{n-1}$, viewed as a connection for $\mathcal{N}$.
\end{enumerate}
\end{definition}

Given Bartnik data, let $\{\nu,\tau\}$ be a framing of $\mathcal{N}$ so that $\nu$ and $\tau$ have squared lengths $+1$ and $-1$.  The correspondence between $\beta$ and its associated connection $\nabla^{\mathcal{N}}$ on $\mathcal{N}$ is given by $\beta(X)=-\langle\nabla_X^{\mathcal{N}}\nu,\tau\rangle$ for tangent vectors $X$ to $\Sigma^{n-1}$. Now consider two such bundles with the same base $\mathcal{N},\mathcal{N}'\to\Sigma^{n-1}$, equipped with framings $\{\nu,\tau\}$ and $\{\nu',\tau'\}$. With respect to these frames, an isometric bundle map $F:\mathcal{N}\to\mathcal{N}'$ may be expressed fiberwise as a family of $SO^+(1,1)$ matrices
\begin{equation}
F|_{\mathcal{N}_x}=\left[\begin{array}{cc}
\cosh(f(x))&-\sinh(f(x))\\
-\sinh(f(x))&\cosh(f(x))
\end{array}\right],\quad\quad\quad x\in\Sigma^{n-1},
\end{equation}
for some uniquely determined $f\in C^\infty(\Sigma^{n-1})$ called the hyperbolic angle of $F$. Two sets of Bartnik data $(\mathcal{N},\vec{H},\beta)$ and $(\mathcal{N}',\vec{H}',\beta')$ on $(\Sigma^{n-1},\gamma)$ are said to be {\emph{equivalent}} if there is an isometric bundle map $F:\mathcal{N}\to\mathcal{N}'$, such that $F\circ\vec{H}=\vec{H}'$ and $\beta'=\beta +df$ where $f$ is the hyperbolic angle of rotation induced by $F$.

A fundamental source of Bartnik data over $(\Sigma^{n-1},\gamma)$ arises from isometric embeddings into initial data sets $(M^n,g,k)$. If $\nu$ is a unit normal vector field to $\Sigma^{n-1}\subset M^n$, then an associated set of Bartnik data may be described as follows. Let $\mathcal{N}$ be the Whitney sum of the normal bundle to $\Sigma^{n-1}$ in $M^n$ together with a trivial line bundle spanned by a vector $\tau$
which we declare to be orthogonal to $\nu$ and of squared length $-1$,
and set
\begin{equation}
\vec{H} = H \nu - (\Tr_\gamma k) \tau, \quad\quad\quad \beta(X)= k(\nu, X),
\end{equation}
where $H$ is the mean curvature of $\Sigma^{n-1}$ with respect to $\nu$ and $X\in T\Sigma^{n-1}$. We will now fix the setting of our main theorem by introducing the necessary geometric condition. 

\begin{definition}\label{def:DECcreased}
An {\emph{asymptotically flat DEC-creased initial data set}} consists of two smooth initial data sets with boundary, an asymptotically flat $(M_+^n,g_+,k_+)$ and a compact $(M_-^n,g_-,k_-)$, satisfying the following properties.
\begin{enumerate}[label=(\roman*), topsep=3pt, itemsep=0ex]
\item The induced boundary manifolds $(\partial M_+^n,g|_{\partial M_+^n})$ and $(\partial M_-^n,g|_{\partial M_-^n})$ are equal to a common manifold $(\Sigma^{n-1},\gamma)$ which inherits Bartnik data $(\mathcal{N}_\pm,\vec{H}_\pm,\beta_\pm)$ 
induced by the unit normals $\nu_+$ pointing into $(M_+^n,g_+,k_+)$ and $\nu_-$ pointing out of $(M_-^n,g_-,k_-)$.
\item There is an isometric bundle map $F:\mathcal{N}_-\to\mathcal{N}_+$ with hyperbolic angle $f$ so that \begin{align}\label{eq:DECCreased}
\left\langle F( \vec{H}_-) - \vec{H}_+, \nu_+ \right\rangle_{\mathcal{N}_+} - \sqrt{ \left\langle F(\vec{H}_-) - \vec{H}_+, \tau_+ \right\rangle_{\mathcal{N}_+}^2 + \left|\beta^\Delta \right|_{\gamma}^2} \geq 0,
\end{align}
where $\{\nu_+,\tau_+\}$ is an $SO^+(1,1)$ frame for $\mathcal{N}_+$ and $\beta^\Delta$ denotes the difference of connections $\beta^\Delta=\beta_+-\beta_--df$.
\end{enumerate}
\end{definition}

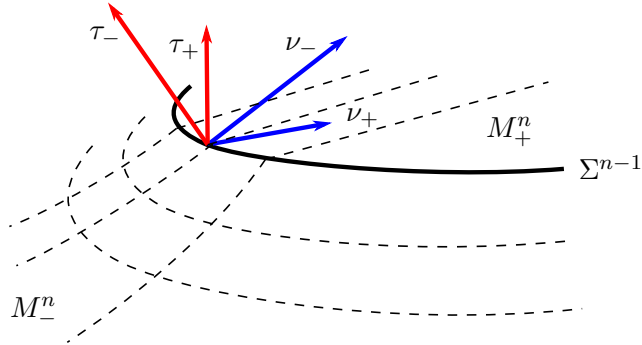
\begin{figure}[h]
\includegraphics[width=3in]{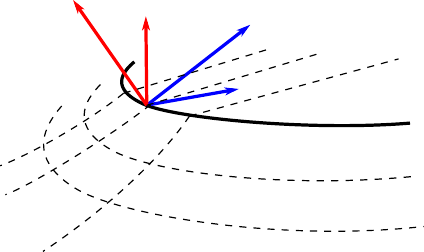}
\begin{picture}(0,0)
\put(-40,78){$M_+^n$}
\put(-5,62){$\Sigma^{n-1}$}
\put(-220,10){$M_-^n$}
\put(-190,115){$\tau_-$}
\put(-160,110){$\tau_+$}
\put(-116,111){$\nu_-$}
\put(-93,84){$\nu_+$}
\end{picture}
\caption{DEC-creased initial data arising from a spacetime.}
    \label{fig:picture}
    \vspace{-.1in}
\end{figure}

In the context of a smooth spacetime, DEC-creased initial data naturally arise from a spacelike hypersurface $M^n$ which is globally Lipschitz and smooth away from an embedded submanifold $\Sigma^{n-1}$, along which $M$ is bent. Given DEC-creased initial data $(M^n_\pm,g_\pm,k_\pm)$, one may glue along $\Sigma^{n-1}$ to form the space $M^n=M^n_-\cup_{\Sigma^{n-1}} M_+^n$, which has a canonical smooth structure using collard neighborhoods \cite{JLee}*{Theorem 9.29}. The metrics $g_-$ and $g_+$ descend to a metric on $M^n$ with regularity $C^0(M^n)\cap C^\infty(M^n\setminus\Sigma^{n-1})$. The union of $k_-$ and $k_+$ does not descend continuously to $M^n$, but may be considered as an $L^\infty$ tensor, with generally different limits on $\Sigma^{n-1}$ as one approaches from $M^n_-$ and $M^n_+$. As will be elucidated through the proof of the main result, the crease condition \eqref{eq:DECCreased} behaves analogously to a weak version of the dominant energy condition along $\Sigma^{n-1}$. The inequality \eqref{eq:DECCreased} is, however, more general than the requirement that $\mu\geq|J|$ holds distributionally along the crease.
We also remark that \eqref{eq:DECCreased} is equivalent to the requirements that $F\circ \vec{H}_--\vec{H}_+$ is spacelike, points in the $\nu_+$ direction, and satisfies 
\begin{equation}
|F(\vec{H}_-)-\vec{H}_+|_{\mathcal{N}_+}\geq|\beta^\Delta|_{\gamma}.
\end{equation}


In Definition \ref{def:DECcreased} there are two trivial framed $SO^+(1,1)$-bundles $\mathcal{N}_\pm$ over $\Sigma^{n-1}$. In this situation, there is always a canonical choice of bundle isometry $F$ which takes $\nu_-$ to $\nu_+$. We will call this map the {\emph{trivial isometry}}.
In the special case where the bundle map $F$ is trivial, the DEC-creased condition $(ii)$ in Definition \ref{def:DECcreased} becomes the ``spacetime corner" condition appearing in \cite{Tsang2022}. When $k\equiv0$ and $F$ is trivial, then the DEC-creased condition becomes the mean curvature inequality $H_-\geq H_+$ that is used in \cite{Miao}.
Our main result is the following positive mass theorem.

\begin{theorem}\label{theorem-main:PMT}
Let $(M_\pm^n,g_\pm,k_\pm)$ be an asymptotically flat DEC-creased initial data set, creased along a closed hypersurface $\Sigma^{n-1}$. Assume that $M^n=M_+^n\cup_{\Sigma^{n-1}} M^n_-$ is spin. If $(M_\pm^n,g_\pm,k_\pm)$ both satisfy the dominant energy condition, then the ADM 4-momentum of any end of $(M_+^n,g_+,k_+)$ is nonspacelike, that is $E\geq |P|$.
\end{theorem}

The proof of Theorem \ref{theorem-main:PMT} relies on a formula for the mass of $(M_+^n,g_+,k_+)$ involving certain spinors defined over $M^n$, stated below as Theorem \ref{theorem-main:existence}. This basic strategy goes back to Witten \cite{Witten}, where one constructs a 
twisted Dirac operator $D_W$, and proceeds by integrating the Lichnerowicz-Schr{\"o}dinger-Weitzenb{\"o}ck formula applied to asymptotically constant spinors satisfying $D_W\psi=0$. In order to deal with the crease singularities, we introduce transmission-type boundary conditions along $\Sigma^{n-1}$ and show that sufficiently regular solutions to the associated boundary value problem exist. To describe this in more detail, 
we first define $\mathrm{Spin}^+(n,1)$ spinor bundles $S_\pm$ over $M_\pm^n$ so that the trivial isometry $\mathcal{N}_-\to\mathcal{N}_+$ induces an isometry $\Phi:S_-|_{\Sigma^{n-1}}\to S_+|_{\Sigma^{n-1}}$.  
The volume form $\epsilon_+$ of $\mathcal{N}_+$ acts as an endomorphism on $S_\pm|_{\Sigma^{n-1}}$ by Clifford multiplication $\epsilon_+=\nu_+\tau_+$. On a given end $M^n_{\ell_1}$ of $M_+^n$, the spin structure can be identified with a trivial spin structure over $\mathbb{R}^n$, giving a source of {\emph{asymptotically constant}} spinors defined over the end. Given a choice of such a spinor $\psi_\infty$, we seek solutions $\psi_\pm\in \Gamma(M^n_\pm,S_\pm)$ to the system 
\begin{equation}\label{eq:mainPDE}
    \begin{cases}
        D_W \psi_+=0&\text{ in }M_+^n ,\\
        D_W \psi_-=0&\text{ in }M_-^n ,\\
        \Phi(\psi_-)=(A+B\epsilon_+)\psi_+&\text{ on }\Sigma^{n-1} ,\\
        |\psi_+-\psi_\infty|(x)\to0&\text{ as }|x|\to\infty \text{ in }M^n_{\ell_1},\\
        |\psi_+|(x)\to0&\text{ as }|x|\to\infty \text{ in }M^n_\ell, \text{ }\ell\neq\ell_1,
    \end{cases}
\end{equation}
where $A=\cosh(f/2)$, $B=\sinh(f/2)$, and $f$ is the hyperbolic angle given by the DEC-creased condition. A geometric interpretation of this system is provided in Remark \ref{rem:boundarycondition} below. We treat \eqref{eq:mainPDE} as a system with boundary conditions, establishing the following existence and regularity statement.

\begin{theorem}\label{theorem-main:existence}
Suppose that $(M_\pm^n,g_\pm,k_\pm)$ is an asymptotically flat DEC-creased initial data set, such that $M^n=M_-^n\cup_{\Sigma^{n-1}}M_+^n$ is spin and the dominant energy condition is satisfied on $M_{\pm}^n$.
Given a constant spinor $\psi_\infty$ on an asymptotically flat end of $M^n_+$, there exists a smooth solution $\psi_\pm\in C^\infty(M_\pm^n,S_\pm)$\footnote{The spinors $\psi_\pm$ are both smooth up to the boundary of $M_\pm^n$, though they do not necessarily define a continuous spinor over $M^n$.} to \eqref{eq:mainPDE} which satisfies
\begin{align}
\begin{split}
\frac{(n-1)\omega_{n-1}}{2}  \left(E|\psi_\infty|^2-\left\langle\psi_\infty,P \tau\psi_\infty\right\rangle\right)\geq &\int_{M_-^n}\left(|\overline{\nabla}\psi_-|^2+\frac12(\mu_-|\psi_-|^2+\left\langle\psi_-,J_-\tau_-\psi_-\right\rangle)\right)dV\\
{}+&\int_{M_+^n}\left(|\overline{\nabla}\psi_+|^2+\frac12(\mu_+|\psi_+|^2+\left\langle\psi_+,J_+\tau_+\psi_+\right\rangle)\right)dV,
\end{split}
\end{align}
where $\tau,\tau_\pm$ are endomorphisms of the spinor bundles which square to the identity.
\end{theorem}


\vspace{12pt}

\noindent{}{\bf{The case of equality.}}  Rigidity statements for the classical positive mass theorem characterize those initial data sets $(M^n,g,k)$ satisfying the dominant energy condition which have an end of vanishing ADM mass. Generally speaking, these statements aim to conclude that $(M^n,g,k)$ arises from an embedding into Minkowski space, which is to say that $(M^n,g)$ is isometrically embedded with second fundamental form equal to $k$. In the time-symmetric setting when $k=0$, Schoen-Yau \cite{SY1} showed that the mass of an end is zero only when the manifold is isometric to Euclidean space for $n\leq 7$. In the nontime-symmetric case, Beig-Chru\'{s}ciel \cite{BeigChrusciel} and Chru\'{s}ciel-Maerten \cite{ChruscielMaerten} considered spin manifolds, imposed additional and necessary decay assumptions, showing that $E=|P|$ implies $E=|P|=0$ and that such $(M^n,g,k)$ must arise as a hypersurface of Minkowski space. Also assuming additional decay assumptions, Huang-Lee \cites{HuangLee,HuangLee2} showed that the desired rigidity statement follows from the positive mass inequality. We also point out the work of Hirsch-Zhang \cites{svenyiyue,HirschZhang}, who give the optimal result in dimension 3 and characterize general spin initial data with vanishing ADM mass in terms of pp-wave spacetimes. In the setting of Theorem \ref{theorem-main:PMT}, we show the following rigidity statement. Note that additional decay beyond the usual asymptotic flatness assumption is imposed, using weighted H{\"o}lder spaces (see \cite{HuangLee}).

\begin{theorem}\label{thm:rigidity}
Suppose that $(M_\pm^n,g_\pm,k_\pm)$ is an asymptotically flat DEC-creased initial data set, such that $M^n=M_-^n\cup_{\Sigma^{n-1}}M_+^n$ is spin and the dominant energy condition is satisfied on $M_{\pm}^n$. Assume there is some $\alpha\in(0,1)$ and $\varsigma>0$ such that 
\begin{equation}
g_+-\delta \in C^{3,\alpha}_{-q}(M^n_+),\quad\quad k_+ \in C^{2,\alpha}_{-q-1}(M^n_+),\quad\quad \mu,J\in C^{1,\alpha}_{-n-\varsigma}(M^n_+),
\end{equation}
where the order of asymptotically flat decay satisfies $q>\max(1/2,n-3)$. If the ADM mass of any end vanishes, then there is Lipschitz embedding $M^n\to \mathbb{R}^{1,n}$ to Minkowski space which is smooth away from $\Sigma^{n-1}$, such that the induced metric and second fundamental form agree with $g_-\cup g_+$ and $k_-\cup k_+$.

\end{theorem}


There are previously known rigidity statements when $m=0$ on an asympotically flat manifold with singularities along a hypersurface $\Sigma^{n-1}$. In the time-symmetric setting where $R_g\geq0$ and $H_-\geq H_+$ along $\Sigma^{n-1}$, Shi-Tam \cite{ShiTam} used spinorial methods to show that $g$ is flat away from $\Sigma$. When $n=3$ and $\Sigma^2\cong S^2$, Miao \cite{Miao} showed there is a $C^{2,1}$ isometry $M^3\to\mathbb{R}^3$. In higher dimensions, using the Ricci flow on manifolds for which mass positivity is known, McFarron-Sz{\'e}kelyhidi \cite{McFeronSzekelyhidi} showed a similar statement but with $C^{2,\alpha}$-regularity, $\alpha\in(0,1)$. To the authors' knowledge, the optimal rigidity statement in this singular time-symmetric setting appears to be unknown in general. In his study of corner singularities (where $F$ is trivial) in dimension 3, Tsang \cite{Tsang2022} proved two rigidity statements: if $E=|P|$ then $M^3$ is diffeomorphic to $\mathbb{R}^3$, and if $E=|P|=0$ and additional regularity is imposed on $k$, then $(M^3,g,k)$ arises from an isometric embedding into Minkowski spacetime and $(g,k)\in C^{2,1}(M^3)\times C^{1,1}(M^3)$. We also point out the recent work of Hirsch-Huang \cite{LanSven}, concerning further results that characterize the case of vanishing mass in the presence of matching Bartnik data.

The proof of Theorem \ref{thm:rigidity} proceeds in three steps. First, the mass formula of Theorem \ref{theorem-main:existence} allows us to obtain a spinor whose corresponding vector field defines a type of Killing lapse-shift pair. Then, the work of Huang-Lee \cite{HuangLee} is used to show that the vanishing of $m$ implies that $E=|P|=0$. This is in turn leveraged to make stronger use of the mass formula, finding many linearly independent Sen-parallel spinors which are used to construct a flat Killing development following Beig-Chrusciel \cite{BeigChrusciel}. 

\vspace{12pt}

\noindent{\bf{Extensions and fill-ins.}} As an application of the creased positive mass theorem, one can establish the nonexistence of certain {\emph{fill-ins}} or {\emph{null-bordisms}} of given Bartnik data arising from an embedding into Minkowski spacetime. 

\begin{corollary}\label{cor:fillin}
Suppose $(\Sigma^{n-1},\gamma)$ embeds isometrically into $\mathbb{R}^{1,n}$ as the boundary of an asymptotically flat spacelike hypersurface with unit inward normal $\nu_0$. Let $(\mathcal{N}_0,\vec{H}_0,\beta_0)$ denote the induced Bartnik data. Assume further that $(\Omega^n, g, k)$ is a compact initial data set with outward normal vector $\nu$ satisfying the following conditions.
\begin{enumerate}[label=(\roman*), topsep=3pt, itemsep=0ex]
\item $\Omega^n$ has a spin structure extending the one induced on $\Sigma^{n-1}$ from the spacelike hypersurface of $\mathbb{R}^{1,n}$.
\item $\mu\geq|J|$ throughout $\Omega$.
\item $(\partial\Omega^n,g|_{\partial\Omega^n})$ is equal to $(\Sigma^{n-1},\gamma)$, inducing a second Bartnik data set $(\mathcal{N},\vec{H},\beta)$.
\item There is a bundle isometry $F:\mathcal{N}\to\mathcal{N}_0$ with hyperbolic angle $f$ so that 
\begin{align}
\left\langle F( \vec{H}) - \vec{H}_0, \nu_0 \right\rangle_{\mathcal{N}_0} - \sqrt{ \left\langle F(\vec{H}) - \vec{H}_0, \tau_0 \right\rangle_{\mathcal{N}_0}^2 + \left|\beta^\Delta \right|_\gamma^2} \geq 0,
\end{align}
where $\beta^\Delta=\beta_0-\beta-df$ and $\{\nu_0,\tau_0\}$ form an $SO^+(1,1)$-framing for $\mathcal{N}_0$.
\end{enumerate}
Then the interior of $(\Omega^n, g, k)$ isometrically embeds into Minkowski spacetime. In particular, there is no $(\Omega^n,g,k)$ satisfying conditions $(i)$--$(iv)$ with $\mu(x)> |J|(x)$ at some point $x\in\Omega^n$.
\end{corollary}

Corollary \ref{cor:fillin} assumes that $(\mathcal{N}_0,\vec{H}_0,\beta_0)$ and $(\Sigma^{n-1},\gamma)$ arise as the boundary of an asymptotically flat spacelike hypersurface in Minkowski spacetime. For example, this condition is satisfied when the spacelike hypersurface is the graph of a function over a constant time slice of $\mathbb{R}^{1,n}$. We also remark that Corollary \ref{cor:fillin} may be viewed as a generalization of Miao's result \cite{Miao}*{Corollary 1.1} concerning fill-ins with nonnegative scalar curvature.


\vspace{12pt}

\noindent{\bf{Relation to quasi-local mass.}} In analogy with the work of Shi-Tam \cite{ShiTam}, we expect that Theorem \ref{theorem-main:PMT} may be a useful tool in studying the positivity properties of the Wang-Yau \cite{WangYaumass} quasi-local mass, though we do not pursue this idea in the present work. There is, however, a more immediate implication to the Bartnik mass of compact initial data sets $(\Omega^n,g,k)$. In Bartnik's approach, one takes the mass of $\Omega^n$ as the infemum of the ADM masses of asymptotically flat extensions $(M^n,g,k)$ of $(\Omega^n,g,k)$ satisfying some admissibility conditions motivated by the positive mass theorem, including the DEC and a `no horizons' provision. Theorem \ref{theorem-main:PMT} suggests that the class of admissible extensions may be widened to include initial data sets with DEC-creases along $\Sigma^{n-1}=\partial\Omega^n$.

\vspace{12pt}

\noindent{\bf{Organization.}} In Section 2, we introduce notation and basic facts related to initial data sets, Bartnik data, and the associated spinor bundles. Section 3 contains the main geometric computation of the boundary terms encountered when integrating the Lichnerowicz-Schr{\"o}dinger-Weitzenb{\"o}ck formula by parts, where the DEC-creased condition appears. The existence of weak solutions to \eqref{eq:mainPDE} is established in Section 4, and concludes with the proof of Theorem \ref{theorem-main:existence}. This proof relies on a regularity result which is postponed to Section 5, where we verify that \eqref{eq:mainPDE} can be interpreted as a elliptic boundary value problem in the formalism of B{\"a}r-Ballman \cite{BaerBallmann}. Finally, Section 6 carries out the proof of the rigidity result Theorem \ref{thm:rigidity}.

\vspace{12pt}

\noindent{\bf{Aknowledgements.}} The authors would like to thank Hubert Bray for many influential discussions and suggestions. We would also like to thank Sven Hirsch and Yiyue Zhang for explaining their work \cite{HirschZhang}.


\section{Bartnik Data and Spinor Bundles}\label{sec:background}

\subsection{Bartnik data}
Throughout this section, we adhere to the context of Theorem \ref{theorem-main:PMT}. In particular, suppose $(M_+^n,g_+, k_+)$ and $(M_-^n,g_-,k_+)$ are $n$-dimensional initial data sets, with $M_+^n$ asymptotically flat and $M_-^n$ compact, and that the boundaries $\partial M_+^n$, $\partial M_-^n$ are equal to a common closed manifold $(\Sigma^{n-1},\gamma)$. 
Let $(\mathcal{N}_\pm, \vec{H}_\pm, \beta_\pm)$ be the Bartnik data on $(\Sigma^{n-1},\gamma)$ arising from $(M_\pm^n, g_\pm, k_\pm)$ using unit normals $\nu_+$ pointing into $M_+^n$ and $\nu_-$ pointing out of $M_-^n$. In particular, $\mathcal{N}_\pm$ has the $SO^+(1,1)$-frame $\{\tau_\pm,\nu_\pm\}$ and
\begin{align}
\begin{split}
\vec{H}_\pm=H_\pm\nu_\pm-(\mathrm{Tr}_\gamma k_\pm) \tau_\pm,\quad\quad\quad \beta_\pm(X)=-\langle\nabla^{\mathcal{N}_\pm}_X\nu_\pm,\tau_\pm\rangle=k_\pm(\nu_\pm,X),\qquad X\in T\Sigma^{n-1},    
\end{split}
\end{align}
where $H_\pm$ denotes the mean curvature of $\Sigma^{n-1}$ as a submanifold of $M_\pm^n$ computed using $\nu_\pm$. 

Using the frames $\{\nu_\pm,\tau_\pm\}$, any bundle isometry $F:\mathcal{N}_-\to\mathcal{N}_+$ can be uniquely described using its hyperbolic angle $f$, viewed as a function on $\Sigma^{n-1}$ and given by
\begin{align}\label{e:tauNuId1}
    F(\nu_-) = \cosh(f)\nu_+ - \sinh(f) \tau_+,\quad\quad\quad 
    F(\tau_-) = - \sinh(f)\nu_+ + \cosh(f) \tau_+.
\end{align}
Given such an $F$, the connection coefficients of $\nabla^{\mathcal{N}_-}$ can be expressed in the frame $\{F^{-1}(\nu_+),F^{-1}(\tau_+)\}$ in the following way
\begin{align}\label{e:oneFormRelation}
\begin{split}
\left\langle \nabla^{\mathcal{N}_-}_X F^{-1}(\nu_+), F^{-1}(\tau_+) \right\rangle
&= \left\langle \nabla^{\mathcal{N}_-}_X \left(\cosh(f) \nu_- + \sinh(f)\tau_-\right),  \sinh(f) \nu_- + \cosh(f)\tau_- \right\rangle \\
&=-\beta_-(X)-df(X).
\end{split}
\end{align}
The difference of connections is frame independent and becomes
\begin{align}
\beta^\Delta (X) := \left\langle \nabla^{\mathcal{N}_-}_XF^{-1}(\nu_+),F^{-1}(\tau_+)\right\rangle - \left\langle \nabla^{\mathcal{N}_+}_X\nu_+,\ \tau_+\right\rangle 
&= \beta_+(X)-\beta_-(X) - df.\label{eq:oneFormDiff}
\end{align}
For convenience, we record the components of $\vec{H}_\pm$ in the $\{\tau_+, \nu_+\}$ frame
\begin{align}
\begin{split}\label{e:HRelation}
\left\langle \vec{H}_+, \nu_+ \right\rangle = H_+, \quad& \quad \left\langle \vec{H}_-, F^{-1}(\nu_+) \right\rangle  = \cosh(f) H_- + \sinh(f) \Tr_{\gamma} k_-, \\
\left\langle \vec{H}_+, \tau_+ \right\rangle = \Tr_{\gamma} k_+,  &\quad \left\langle  \vec{H}_-, F^{-1}(\tau_+) \right\rangle = \sinh(f) H_- + \cosh(f) \Tr_{\gamma} k_-.
\end{split}
\end{align}

\subsection{Spin structures}
Let $P^{SO}_\pm$ denote be the bundle of oriented orthonormal frames over $(M_\pm^n, g_\pm)$. Define a map $\phi_0:TM^n_-|_{\Sigma^{n-1}}\to TM_+^n|_{\Sigma^{n-1}}$ sending $\nu_-$ to $\nu_+$ and acting by the identity on $T{\Sigma^{n-1}}$.
We also regard $\phi_0$ as a map  $P^{SO}_-|_{\Sigma^{n-1}} \to P^{SO}_+|_{\Sigma^{n-1}}$. 
Since we assume that the smooth manifold $M^n=M_-^n\cup_{\Sigma^{n-1}} M_+^n$ is spin, there are 
$\mathrm{Spin}(n)$-reductions $P_-^{\mathrm{Spin}} \to P^{SO}_-$ and $P_+^{\mathrm{Spin}} \to P_+^{SO}$ which agree over $\Sigma^{n-1}$ in the sense that there is a bundle map $\Phi_0$ making the following diagram commute
\begin{equation}\label{PhiDef}
    \begin{tikzcd}
   P_-^{\mathrm{Spin}}|_{\Sigma^{n-1}}  \arrow[r, "\Phi_0"] \arrow[d] &  P_+^{\mathrm{Spin}}|_{\Sigma^{n-1}} \arrow[d] \\
  P^{SO}_-|_{\Sigma^{n-1}} \arrow[r, "\phi_0" ] & P^{SO}_+|_{\Sigma^{n-1}}
\end{tikzcd}.
\end{equation}

Making use of the spin representation, we obtain associated spinor bundles $S^\pm_0$ over $M_\pm^n$, equipped with Clifford multiplication, a hermitian metric $\left\langle \cdot, \cdot \right\rangle_0$, and a compatible spin connection $\nabla_0$. Interpreting $P^{\mathrm{Spin}}_\pm$ as the bundle of orthonormal frames of $S^\pm_0$, $\Phi_0$ induces an isometry $ S^-_0|_{\Sigma^{n-1}} \to  S^+_0|_{\Sigma^{n-1}}$ which we continue to denote by $\Phi_0$. By \eqref{PhiDef}, the map $\Phi_0$ respects Clifford multiplication in the sense that
$\Phi_0(v \psi) = \phi_0(v) \Phi_0(\psi)$ for all $v \in TM_-^n|_{\Sigma^{n-1}}$ and spinors $\psi$. In particular, $\Phi_0(\nu_- \psi) = \nu_+ \Phi_0(\psi)$ and $\Phi_0(X \psi) = X \Phi_0(\psi)$ for $X \in T\Sigma^{n-1}$.

Following \cite{Lee}*{Section 8.3}, we construct the {\emph{spacetime spinor bundle} $S_\pm:=S^\pm_0\oplus S^\pm_0$. Recall that the bundles $\mathcal{N}_\pm$ are defined as the Whitney sum of the normal bundle from the embedding $\Sigma^{n-1}\subset M_\pm^n$ with a trivial bundle spanned by $\tau_\pm$. Extend the trivial bundle spanned by $\tau_\pm$ over $M_\pm^n$, and equip the sum $\mathrm{span}(\tau_\pm)\oplus TM_\pm^n$ with a Lorentzian bundle metric $h_\pm$ defined by declaring $\tau_\pm$ to be orthogonal to $TM_\pm^n$ and have square norm $-1$. The associated Clifford bundle $\mathrm{Cl}(\mathrm{span}(\tau_\pm)\oplus TM_\pm^n)$ acts on sections $\psi=\psi_1\oplus\psi_2$ of $S_\pm$ as follows:
\begin{align}
X\psi=X\psi_1\oplus -X\psi_2\quad \text{ for } X\in TM_\pm^n,\quad \quad\quad \quad \tau_\pm\psi=\psi_2\oplus\psi_1.
\end{align}
In our convention, Clifford multiplication satisfies $VW+WV=-2h_\pm(V,W)$. Note that $\tau_\pm$ satisfies $\tau_\pm^2 = 1$, and anticommutes with Clifford multiplication by elements of $TM_\pm^n$.

The spinor bundles $S_\pm$ carry two inner products, namely $\langle\cdot,\cdot\rangle$ and $(\cdot,\cdot)$. The former is positive definite and satisfies
\begin{equation}\label{eq:<>inv}
\langle X\psi,X\psi'\rangle=|X|^2\langle\psi,\psi'\rangle\quad \text{ for } X\in TM_\pm^n\qquad\qquad
    \langle \tau_\pm\psi,\tau_\pm\psi'\rangle=\langle\psi,\psi'\rangle.
\end{equation}
The pairing $(\cdot,\cdot)$ 
satisfies the first equality of \eqref{eq:<>inv} and the second equality with a minus sign on the right-hand side. Along $\Sigma^{n-1}$, the bundle map $\Phi_0$ induces an bundle isometry $\Phi:S_-|_{\Sigma^{n-1}}\to S_+|_{\Sigma^{n-1}}$ with respect to the pairing $\langle\cdot,\cdot\rangle$. In what follows, it is useful to introduce local framings for $S_\pm$. Given a local orthonormal frame $\{e_i\}_{i=1}^n$ for $(M_\pm^n,g_\pm)$, the spin structure on $M_\pm^n$ allows one to lift $\{e_i\}_{i=1}^n$ to a local orthonormal frame of $S_0^\pm$. Duplicating this frame gives rise to a framing $\{\psi_l\}_{l=1}^I$ of $S_\pm$ where $I=\mathrm{dim}(S_\pm)$. In this case, we we refer to $\{\psi_l\}_{l=1}^I$ as {\emph{induced by $\{e_i\}_{i=1}^n$}}.

\vspace{12pt}

\noindent{\bf{Connections and Dirac operators.}} There are several natural connections on the spinor bundles $S_\pm$ and $S_\pm|_{\Sigma^{n-1}}$ which we will now identify and describe:
\begin{itemize}
    \item the {\emph{$M_\pm^n$ Riemannian connections}} $\nabla^\pm$ on $S_\pm$ from the spin connections induced by the Levi-Civita connection of $g_\pm$ on $S_0^\pm$,
    \item the {\emph{Sen (or spacetime) connections}} $\overline{\nabla}^\pm$ on $S_\pm$ defined in a local frame by
    \begin{equation}
        \overline{\nabla}^\pm_i\psi=\nabla^{\pm}_i\psi+\frac12 (k_\pm)^j_ie_j\tau_\pm\psi,
    \end{equation}
    \item the {\emph{boundary spin connection}} $\nabla^{\Sigma,\pm}$ on $S_\pm|_{\Sigma^{n-1}}$ induced by the Levi-Civita connection of $\gamma$.
\end{itemize}
It will be convenient to have a local formula for $\nabla^{\Sigma,\pm}$. Fix a local frame $\{\psi^\pm_l\}_{l=1}^I$ of $S_{\pm}|_{\Sigma^{n-1}}$ induced by a frame $\{e_1, \dots, e_{n-1}, \nu_\pm \}$, where $\{e_\alpha\}_{\alpha=1}^{n-1}$ is an orthonormal frame for ${\Sigma^{n-1}}$. If $\psi = \sum_{l=1}^I c^\pm_l \psi^\pm_l$ is a section of $S_{\pm}|_{\Sigma^{n-1}}$, then we have
\begin{align}\label{eq:BoundaryConnection}
    \nabla^{\Sigma,\pm}_{X} \psi = \sum_{l=1}^I X(c^\pm_l) \psi^\pm_l - \frac14 \sum_{l=1}^I\sum_{\alpha,\eta=1}^{n-1}c^\pm_l\omega_{\alpha\eta}(X)e_\alpha e_\eta \psi^\pm_l
\end{align}
for any $X \in T\Sigma^{n-1}$, where $\omega_{\alpha\eta}(X) = \left\langle e_\alpha, \nabla_X e_\eta\right\rangle$. 
The following result summarizes the relevant properties used in subsequent computations.

\begin{proposition}\label{prop:Phiproperties}
The connections $\nabla^\pm$, $\overline{\nabla}^\pm$, $\nabla^{\Sigma,\pm}$, the pairings $\langle\cdot,\cdot\rangle$, $(\cdot,\cdot)$, the bundle isometry $\Phi$, and the endormophisms $\tau_\pm$ defined above satisfy the following properties.
\begin{enumerate}[label=(\roman*), topsep=3pt, itemsep=0ex]
\item $\nabla^\pm$ is compatible with $\langle\cdot,\cdot\rangle$ and $\overline{\nabla}^\pm$ is compatible with $(\cdot,\cdot)$.
\item For any $\psi\in S_-|_{\Sigma^{n-1}}$, we have
\begin{align}
\begin{split}
\Phi(X \psi) &= X\Phi(\psi) \quad\quad\textrm{ for }X \in T\Sigma^{n-1}, \\ 
\Phi(\nu_- \psi) &= \nu_+ \Phi(\psi),\quad    \Phi(\tau_- \psi) = \tau_+ \Phi(\psi) .
\end{split}
\end{align}
\item $\Phi$ and $\tau_\pm$ are parallel with respect to $\nabla^{\Sigma,\pm}$ and $\nabla^\pm$, that is $\nabla^{\Sigma,+} \circ \Phi = \Phi \circ \nabla^{\Sigma,-}$ and $\nabla^\pm\circ\tau_\pm=\tau_\pm\circ\nabla^\pm.$
\end{enumerate}
\end{proposition}

\begin{proof}
Item $(i)$ and the statement $\nabla^\pm\circ\tau_\pm=\tau_\pm\circ\nabla^\pm$ of item $(iii)$ are standard, see for instance \cite{BartnikChrusciel}. Item $(ii)$ follows from the fact that $\phi_0$ acts by the identity on $T\Sigma^{n-1}$, and rotates the frame $\{\nu_-,\tau_-\}$ into $\{\nu_+,\tau_+\}$. 

To prove the statement $\nabla^{\Sigma,+} \circ \Phi = \Phi \circ \nabla^{\Sigma,-}$ of item $(iii)$, let $\{\psi^\pm_l\}_{l=1}^I$ be spin frames induced by an orthonormal frame $\{e_1,\dots,e_{n-1},\nu_\pm\}$ for $TM_\pm^n|_{\Sigma^{n-1}}$. The diagram \eqref{PhiDef} implies that $\{\psi^+_l\}_{l=1}^I$ can be chosen as $\{\Phi(\psi^-_l)\}_{l=1}^I$, so that if $\psi=\sum_{l=1}^I c^-_l\psi^-_l$ then $\Phi(\psi)=\sum_{l=1}^I c^-_l\psi^+_l$. Hence we have
\begin{align}
\begin{split}
\Phi \left(\nabla^{\Sigma,-}_{X} \psi \right)&= \Phi \left( \sum_{l=1}^I X(c_l^-) \psi_l^- - \frac14 \sum_{l=1}^I\sum_{\alpha,\eta=1}^{n-1}c_l^-\omega_{\alpha\eta}(X)e_\alpha e_\eta\psi_l^- \right) \\
&=   \sum_{l=1}^I X(c_l^-) \Phi (\psi_l^-) - \frac14 \sum_{l=1}^I\sum_{\alpha,\eta=1}^{n-1}c_l^-\omega_{\alpha\eta}(X)e_\alpha e_\eta  \Phi(\psi_l^-)\\
&=  \nabla^{\Sigma,+}_X \Phi(\psi).
\end{split}
\end{align}
\end{proof}

The various connections give rise to certain Dirac operators. Fix local orthonormal frame $\{e_i\}_{i=1}^n$ for $(M^n_{\pm},g_{\pm})$, with $e_n=\nu_\pm$ on $\Sigma^{n-1}$. Roman indicies will run over $i=1,\dots, n$ and Greek indices will run over $\alpha=1,\dots, n-1$. The {\emph{Dirac-Witten operator}} is given by
\begin{equation}
D_W=e^i\overline{\nabla}_i ,
\end{equation}
and the {\emph{boundary Dirac operator}} is given by
\begin{equation}
\mathcal{D}_\pm=\mp\nu_\pm e^\alpha\nabla^{\Sigma,\pm}_\alpha,
\end{equation}
where the covectors $e^i$, $e^\alpha$ act by Clifford multiplication after lowering the index.
Note the above sign change, which adheres to the convention that $\mp\nu_\pm$ points out of $M_\pm^n$.


\subsection{The Lichnerowicz-Schr{\"o}dinger-Weitzenb\"ock formula.}
To properly introduce the fundamental geometric identity, fix a general smooth compact spin initial data set $(\Omega^n,g,k)$. We consider a section $\psi$ of the spacetime spinor bundle with timelike endomorphism $\tau$, the Sen connection $\overline{\nabla}$, Dirac-Witten operator $D_W$, outward normal $\nu$, boundary mean curvature $H$, and boundary Dirac operator $\mathcal{D}^{\partial\Omega}$. By integrating the Lichnerowicz-Schr{\"o}dinger-Weitzenb\"ock formula by parts for a spinor $\psi$, one may obtain \cite{BartnikChrusciel}*{(11.13)} the identity
\begin{align}
\begin{split}\label{eq:Lichnerowicz}
&\quad \int_{\Omega^n} \left(|\overline\nabla \psi|^2 - |D_W \psi|^2+\frac12\left\langle \psi,\ (\mu+J\tau)\psi \right\rangle \right)dV \\
{}&=\int_{\partial\Omega^n} \left\langle\psi,\mathcal{D}^{\partial\Omega}\psi-\frac12H\psi-\frac12\left[\left(\Tr_{\partial\Omega}k\right)\nu-k(\nu,e_\alpha)e^\alpha \right] \tau \psi \right\rangle dA
\end{split}
\end{align}
where $\{e_\alpha\}_{\alpha=1}^{n-1}$ is an orthonormal frame for $\partial \Omega^n$.



\vspace{12pt}

\noindent{\bf{Constant spinors at infinity and mass.}}
Here and throughout, let $q_* = \frac{n-2}{2}$. For an $n$-dimensional asymptotically flat spin initial data set $(M^n,g,k)$ with spacetime spinor bundle $S$, define a weighted Sobolev space $W^{1,2}_{-q_*}(S)$ as the completion of $C_c^\infty(S)$ with respect to the norm
\begin{align}
    \|\psi\|_{W^{1,2}_{-q_*}(S)}^2 := \|\nabla \psi \|_{L^2(S)}^2 + \|\psi /r\|_{L^2(S)}^2,
\end{align}
where $r$ is a positive extension to all $M^n$ of the radial coordinate $r = |x|$ in the asymptotically flat end.


Fix an orthonormal frame $\{e_i\}_{i=1}^n$ for $(M^n,g)$ over an asymptotically flat end $M^n_\ell$ which is asymptotic to the coordinate vector fields $\{\partial_i\}_{i=1}^n$ in the sense that $|e_i-\partial_i|=O(r^{-q_*})$ for all $i$. Let $\{\psi_l\}_{l=1}^I$ denote an associated frame for the spacetime spinor bundle over $M^n_\ell$. A spacetime spinor \textit{$\psi_\infty$} is called a \textit{constant spinor at infinity} if $\psi_\infty = \sum_{l=1}^I c_l \psi_l$ for constant functions $c_l$. We say a spacetime spinor $\psi$ is {\emph{asymptotic}} to such a constant spinor at infinity if $\psi - \psi_\infty \in W^{1,2}_{-q_*}(S|_{M^n_\ell})$. In this case, we use the notation $\psi\to\psi_\infty$ on $M^n_\ell$.
The ADM energy and momenta can be computed in terms of asymptotically constant spacetime spinors. In particular, Witten \cite{Witten} observed the following.

\begin{proposition}[\cite{BeigChrusciel}]\label{prop:boundarylimit}
Suppose $(M^n,g,k)$ is a spin asymptotically flat initial data set with a spacetime spinor $\psi$ asymptotic to a constant spinor $\psi_\infty$. Then
\begin{align}
\begin{split}
& \lim_{r \to \infty} \int_{S_r} \left\langle\psi_\infty,\ \mathcal{D}^{S_r}\psi_\infty-\frac12H\psi_\infty-\frac12\left[\left(\Tr_{S_r}k\right)\upsilon-k(\upsilon,\cdot) \right] \tau \psi_\infty \right\rangle dA \\
&= \frac{(n-1)\omega_{n-1}}{2}  \left(E|\psi_\infty|^2-\left\langle\psi_\infty,P_ie^i \tau\psi_\infty\right\rangle\right),
\end{split}
\end{align}
where $\mathcal{D}^{S_r}$ and $H$ denote the boundary Dirac operator and mean curvature of the coordinate sphere $S_r$ with respect to the outward unit normal $\upsilon$.
\end{proposition}

\section{Spinor Calculations}\label{sec:spinorcalculation}

Suppose $(M_\pm^n,g_\pm,k_\pm)$ is an asymptotically flat DEC-creased initial data set. Let $f$ be the hyperbolic angle associated to the bundle isometry $F:\mathcal{N}_-\to\mathcal{N}_+$, set
\begin{align}
a=\cosh(f), \quad b=\sinh(f),\quad A=\cosh(f/2), \quad B=\sinh(f/2),
\end{align}
and let $\epsilon_\pm=\nu_\pm \tau_\pm$. Acting by Clifford multiplication, $\epsilon_\pm$ is an isometry of $(S_\pm|_{\Sigma^{n-1}},\langle\cdot,\cdot\rangle)$ due to \eqref{eq:<>inv}, and satisfies 
the following properties which are frequently used below.

\begin{proposition}\label{prop:epsilonidentities}
The endomorphism $\epsilon_\pm$ and functions $a,b,A,B$ defined above satisfy 
    \begin{align}\label{eq:epsilonproperties}
        \epsilon_\pm^2 = 1, \quad\quad \epsilon_\pm \nu_\pm = - \nu_\pm \epsilon_\pm = \tau_\pm, \quad\quad \epsilon_\pm \tau_\pm = - \tau_\pm \epsilon_\pm, \quad\quad \epsilon_\pm \circ \nabla^{\Sigma,\pm} = \nabla^{\Sigma,\pm} \circ \epsilon_\pm,
    \end{align}
and
    \begin{align}\label{eq:abABproperties}
        (A-B\epsilon_\pm)(A+B\epsilon_\pm)=1,\quad (A+B\epsilon_\pm)^2=a+b\epsilon_\pm.
    \end{align}
\end{proposition}

\begin{proof}
    The identities \eqref{eq:epsilonproperties} follow from the Clifford relations and part $(iii)$ of Proposition \ref{prop:Phiproperties}. To establish \eqref{eq:abABproperties}, one applies the hyperbolic identity $a^2-b^2=A^2-B^2=1$ and the double angle formulas $A^2+B^2=a$, $2AB=b$.
\end{proof}

\begin{remark}\label{rem:boundarycondition}
    To explain the boundary condition $\Phi(\psi_-)=(A+B\epsilon_+)\psi_+$, suppose $(M_\pm^n,g_\pm,k_\pm)$ arises as a spacelike Lipschitz-embedded hypersurface $M^n$ in a spacetime $(N^{n+1},h)$. Hyperbolic rotation by $f$ on $TN^{n+1}|_{\Sigma^{n+1}}$ taking $\{\nu_-,\tau_-\}$ to $\{\nu_+,\tau_+\}$ induces the rotation $(A+B\epsilon_+)$ on $S_+|_{\Sigma^{n-1}}$, which can be seen using \eqref{eq:abABproperties}. With this in mind, the boundary condition requires that the identification $\Phi:S_-|_{\Sigma^{n-1}}\to S_{+}|_{\Sigma^{n-1}}$ takes $\psi_-$ to the the rotation of $\psi_+$ by $f$. This condition implies that the piecewise defined spinor $\psi_-\cup\psi_+$ is a continuous section of $N^{n+1}$'s spinor bundle restricted to $M^n$.
\end{remark}

\subsection{Boundary term computation}

Given spacetime spinors $\psi_\pm$, denote the boundary terms along $\Sigma^{n-1}$ in the Lichnerowicz-Schr{\"o}dinger-Weitzenb{\"o}ck formula \eqref{eq:Lichnerowicz} by
\begin{align}
\begin{split}
    \mathbb{I}_- &:= \int_{\Sigma^{n-1}} \left\langle\psi_- ,\ \mathcal{D}_- \psi_- - \frac12 H_- \psi_- - \frac12\left[\left(\Tr_\gamma k_- \right)\nu_- - k_-(\nu_-,\cdot) \right] \tau_- \psi_-\right \rangle dA, \\
\mathbb{I}_+ &:= \int_{\Sigma^{n-1}} \left\langle\psi_+,\ \mathcal{D}_+ \psi_+ + \frac12 H_+ \psi_+ +\frac12\left[\left(\Tr_\gamma k_+ \right)\nu_+ - k_+(\nu_+,\cdot) \right] \tau_+ \psi_+\right \rangle dA.
\end{split}
\end{align}
We will show that $\mathbb{I}_- + \mathbb{I}_+$ has a favorable sign under appropriate boundary conditions for $\psi_\pm$ and the DEC-creased condition.

\begin{proposition}\label{p:BoundaryTerms}
    Let $(M_\pm^n,g_\pm,k_\pm)$ and $F$ be as in Theorem \ref{theorem-main:PMT}. Suppose $\psi_\pm$ are spacetime spinors satisfying $\Phi(\psi_-)=(A+B\epsilon_+)\psi_+$ along $\Sigma^{n-1}$. Then
    \begin{align}\label{eq:boundaryterminequality}
    \mathbb{I}_-+\mathbb{I}_+\leq \frac12\int_{\Sigma^{n-1}}|\psi_+|^2 \left[ \left\langle \vec{H}_+ - F(\vec{H}_-),  \nu_+ \right\rangle+ \sqrt{ \left\langle F(\vec{H}_-) - \vec{H}_+, \tau_+ \right\rangle^2 + \left| \beta^\Delta \right|_{\gamma}^2} \right] dA.
    \end{align}
    In particular, the DEC-creased condition ensures $\mathbb{I_-}  + \mathbb{I_+} \leq 0$.
\end{proposition}

\begin{proof}
We will show that $\mathbb{I}_-+\mathbb{I}_+$ takes the form
\begin{align}\label{eq:boundarytermequality}
\begin{split}
\frac12\int_{\Sigma^{n-1}}\left(|\psi_+|^2 \left\langle \vec{H}_+ - F(\vec{H}_-),  \nu_+ \right\rangle+ \left\langle\psi_+,\ \left( \left\langle \vec{H}_+ - F(\vec{H}_-),\tau_+ \right\rangle \nu_+ - \beta^\Delta \right) \tau_+ \psi_+\right\rangle \right)dA,
\end{split}
\end{align}
which implies inequality \eqref{eq:boundaryterminequality} with the help of Cauchy-Schwartz. The computation proceeds by computing each term in the integrand of $\mathbb{I}_-$ entirely in terms of $(M_+^n,g_+,k_+)$ data. 

For the the boundary Dirac operator term note that 
\begin{align}\label{eq:PhiCommuteDirac}
    \Phi( \mathcal{D}_-\psi_-)&=\Phi(\nu_-e^\alpha\nabla^{\Sigma,-}_\alpha\psi_-)=\nu_+e^\alpha\nabla^{\Sigma,+}_\alpha\Phi(\psi_-)=-\mathcal{D}_+\Phi(\psi_-),
\end{align}
where we have used Proposition \ref{prop:Phiproperties}.
Since $\Phi$ is an isometry, it follows that 
\begin{equation}\label{oafnpoianpih}
\left\langle\psi_-,\ \mathcal{D}_-\psi_- \right\rangle=-\left\langle\Phi(\psi_-),\ \mathcal{D}_+\Phi(\psi_-)\right\rangle.
\end{equation}
Before continuing, note the elementary computations $dA=\frac12 Bdf$, $dB=\frac12 Adf$, and $\left\langle\phi,\ (s+t\epsilon_\pm)\psi\right\rangle=\left\langle(s+t\epsilon_\pm)\phi,\ \psi\right\rangle$ for any numbers $s$, $t$, where we have used that $\epsilon_\pm^2=1$.  
By employing the boundary conditions \eqref{eq:mainPDE} and Proposition \ref{prop:epsilonidentities}, it holds that
\begin{align}
\begin{split}
    \left\langle\Phi(\psi_-),\ \mathcal{D}_+\Phi(\psi_-)\right\rangle&=\left\langle(A+B\epsilon_+)\psi_+,\ \mathcal{D}_+(A+B\epsilon_+)\psi_+\right\rangle\\
    {}&=-\left\langle(A+B\epsilon_+)\psi_+,\ \nu_+e^\alpha\nabla_\alpha^{\Sigma,+}(A+B\epsilon_+)\psi_+\right\rangle\\
    {}&=-\left\langle(A+B\epsilon_+)\psi_+,\ \nu_+(dA+dB\epsilon_+)\psi_+-(A-B\epsilon_+)\mathcal{D}_+\psi_+\right\rangle\\
    {}&=-\frac12 \left\langle(A+B\epsilon_+)\psi_+,\ \nu_+(A+B\epsilon_+)df\psi_+\right\rangle+  \left\langle\psi_+,\mathcal{D}_+\psi_+\right\rangle\\
    {}&=-\frac12 \left\langle (A+B\epsilon_+)\psi_+,\ (B-A\epsilon_+)\nu_+df\psi_+\right\rangle+\left\langle\psi_+,\ \mathcal{D}_+\psi_+\right\rangle\\
    {}&=\frac12 \left\langle(A^2-B^2)\epsilon_+\psi_+,\ \nu_+df\psi_+\right\rangle+  \left\langle\psi_+,\ \mathcal{D}_+\psi_+\right\rangle\\
    {}&=\frac12 \left\langle\psi_+,\ \epsilon_+\nu_+df\psi_+\right\rangle+  \left\langle\psi_+,\ \mathcal{D}_+\psi_+\right\rangle\\
    {}&=\frac12 \left\langle\psi_+,\ \tau_+df\psi_+\right\rangle+  \left\langle\psi_+,\ \mathcal{D}_+\psi_+\right\rangle,
\end{split}
\end{align}
and hence together with \eqref{oafnpoianpih} we have
\begin{align}
    \left\langle \psi_-,\ \mathcal{D}_-\psi_-\right\rangle+\left\langle \psi_+,\ \mathcal{D}_+\psi_+\right\rangle= \frac12\left\langle\psi_+,
    \ df\tau_+\psi_+\right\rangle.\label{eq:diracterms}
\end{align}

Next, consider the boundary terms involving the connection 1-form. Similar manipulations produce
\begin{align}
\begin{split}
    \left\langle\psi_-,\ k_-(\nu_-,\cdot)\tau_-\psi_-\right\rangle&=\left\langle\Phi(\psi_-),\ \Phi \left(k_-(\nu_-,\cdot)\tau_-\psi_-\right)\right\rangle \\
    {}&=\left\langle\Phi(\psi_-),\ k_-(\nu_-,\cdot)\tau_+\Phi(\psi_-)\right\rangle \\
    &=\left\langle(A+B\epsilon_+)\psi_+,\ k_-(\nu_-,\cdot)\tau_+(A+B\epsilon_+)\psi_+\right\rangle \\
    &= \left\langle(A+B\epsilon_+)\psi_+,\ (A-B\epsilon_+) k_-(\nu_-,\cdot)\tau_+\psi_+\right\rangle \\
    &= \left\langle \psi_+,\ k_-(\nu_-,\cdot)\tau_+\psi_+\right\rangle ,
\end{split}
\end{align}
where we have used the fact that $\tau_+$ anti-commutes with $\epsilon_+$. The change of gauge formula \eqref{eq:oneFormDiff} now implies 
\begin{align}
\begin{split}
    \frac12\left\langle\psi_-,\ k_-(\nu_-,\cdot)\tau_-\psi_-\right\rangle-\frac12\left\langle\psi_+,\ k_+(\nu_+,\cdot)\tau_+\psi_+\right\rangle&=\frac12\left\langle\psi_+,\ \left(k_-(\nu_-,\cdot) -  k_+(\nu_+,\cdot) \right)\tau_+\psi_+\right\rangle \\
    &= - \frac12\left\langle\psi_+,\ \beta^\Delta \tau_+\psi_+\right\rangle - \frac12\left\langle\psi_+,\ df \tau_+\psi_+\right\rangle.
\end{split}
\label{eq:connectionterm}
\end{align}
Consider now the terms involving the mean curvatures, and observe that Proposition \eqref{prop:epsilonidentities} yields 
\begin{align}
\begin{split}
    \left\langle\psi_-,\ (H_-+\left(\Tr_\gamma k_-\right)\nu_-\tau_-)\psi_-\right\rangle{}&=\left\langle\Phi(\psi_-),\ \Phi((H_-+\left(\Tr_\gamma k_-\right)\nu_-\tau_-)\psi_-)\right\rangle\\
    {}&=\left\langle\Phi(\psi_-),\ (H_- + \left(\Tr_\gamma k_-\right)\nu_+\tau_+)\Phi(\psi_-)\right\rangle\\
    {}&=\left\langle(A+B\epsilon_+)\psi_+,\ (H_- + \left(\Tr_\gamma k_-\right)\epsilon_+)(A+B\epsilon_+)\psi_+\right\rangle\\
    {}&=\left\langle(A+B\epsilon_+)\psi_+,\ (A+B\epsilon_+)(H_- + \left(\Tr_\gamma k_-\right)\epsilon_+)\psi_+\right\rangle\\
    {}&=\left\langle\psi_+,\ (A+B\epsilon_+)^2(H_- + \left(\Tr_\gamma k_-\right)\epsilon_+)\psi_+\right\rangle\\
    {}&=\left\langle\psi_+,\ (a+b\epsilon_+)(H_- + \left(\Tr_\gamma k_-\right)\epsilon_+)\psi_+\right\rangle\\
    {}&= \left\langle \psi_+,\ \left(aH_- + b\Tr_\gamma k_-+\left(bH_-+a\Tr_\gamma k_-\right)\epsilon_+ \right)\psi_+\right\rangle.
\end{split}
\end{align}
It follows that 
\begin{align}
\begin{split}
    &-\left\langle\psi_-,\ \frac12 \left(H_- + \left(\Tr_\gamma k_-\right)\nu_-\tau_- \right)\psi_-\right\rangle + \left\langle\psi_+,\ \frac12 \left(H_+ + \left(\Tr_\gamma k_+\right)\nu_+\tau_+ \right)\psi_+\right\rangle\\
    {}&=\frac12|\psi_+|^2 \left[H_+ - \left(aH_-+b\left(\Tr_\gamma k_-\right)\right)\right]+\frac12\left\langle\psi_+,\epsilon_+\psi_+\right\rangle \left[\Tr_\gamma k_+ -\left(bH_-+a\left(\Tr_\gamma k_-\right)\right) \right] \\
    &= \frac12 |\psi_+|^2 \left\langle \vec{H}_+ - F(\vec{H}_-),  \nu_+ \right\rangle + \frac12 \left\langle\psi_+,\ \epsilon_+\psi_+\right\rangle \left\langle \vec{H}_+ - F(\vec{H}_-),\tau_+ \right\rangle \\
    &= \frac12 |\psi_+|^2 \left\langle \vec{H}_+ - F(\vec{H}_-),  \nu_+ \right\rangle + \frac12 \left\langle\psi_+,\ \left\langle \vec{H}_+ - F(\vec{H}_-),\tau_+ \right\rangle \nu_+ \tau_+ \psi_+\right\rangle,
\end{split}\label{eq:meancurvdiff}
\end{align}
where we have used \eqref{e:HRelation}. Summing \eqref{eq:diracterms}, \eqref{eq:connectionterm}, and \eqref{eq:meancurvdiff} gives the desired formula \eqref{eq:boundarytermequality}.
\end{proof}
    

\section{Existence of Solutions to the Transmission Boundary Value Problem}

In this section we will study the system \eqref{eq:mainPDE}, where $(M_\pm^n,g_\pm,k_\pm)$ are as in Theorem \ref{theorem-main:PMT}. Weighted spaces and associated Poincar\'e inequalities play an important role in the existence theory of Dirac harmonic spinors in the asymptotically flat setting, see \cite{BartnikChrusciel} and \cite{ParkerTaubes}. For convenience we will set $M' = M_+^n \sqcup M_-^n$ to be the \textit{disjoint} union of $M_+^n$ and $M_-^n$, so that $\partial M' = \Sigma^{n-1} \sqcup \Sigma^{n-1}$. We write $W^{1,2}_{-q_*}(M')$ for spinors $\psi$ over $M'$ consisting of pairs $\psi_\pm\in W^{1,2}_{-q_*}(S_\pm)$, and similarly for other functions spaces such as $C_c^\infty(M')$. 
To incorporate the boundary conditions, define the Hilbert space  
\begin{align}
    \mathcal{H} := \left\{ \psi \in W^{1,2}_{-q_*}(M')  \ \Big|\  (A + B\epsilon_+ )\psi_+|_{\Sigma^{n-1}} =  \Phi(\psi_-|_{\Sigma^{n-1}})\right\},
\end{align}
where $q_{*}=\frac{n-2}{2}$. Since the trace operation $\mathcal{T}:W^{1,2}(M')\rightarrow H^{\frac{1}{2}}(\partial M')$ is continuous, $\mathcal{H}$ is indeed a closed subspace of $W^{1,2}_{-q_*}(M')$ and hence a Hilbert space. This space admits approximation by smooth spinors with compact support. 

\begin{lemma}\label{l:smoothApproxH}
$\mathcal{H}\cap C_c^\infty(M')$ is dense in $\mathcal{H}$.
\end{lemma}

\begin{proof}
Take $\psi \in \mathcal{H}$, and let $\psi'\in C_c^\infty(M')$ be close to $\psi$ in $W^{1,2}_{-q_*}(M')$. The approximation $\psi'$ may not satisfy the desired boundary conditions, and thus $\psi'$ may not lie in $\mathcal{H}$. To rectify this situation, consider a bounded right-inverse $\mathcal{E}: H^{\frac{1}{2}}(\partial M') \to W^{1,2}_{-q_*}(M')$ for the trace operator $\mathcal{T}$; this may be constructed to send smooth spinors on $\partial M'$ to smooth spinors on $M'$ which vanish on the ends of $M_+^n$. Define $\mathcal{K}: H^{\frac12}(\partial M') \to  H^{\frac12}(\partial M')$ by
\begin{align}\label{eq:calK}
\varphi_+\oplus \varphi_- \mapsto \frac12 \left(\varphi_+ - (A- B\epsilon_+ ) \Phi(\varphi_-) \right)\oplus\frac12 \left( \varphi_- - (A+B\epsilon_- )\Phi^{-1}(\varphi_+)\right).
\end{align}
Note that $\mathcal{K} \mathcal{T} \varphi = 0$ if and only if $\varphi \in \mathcal{H}$, and that $\mathcal{K}^2 = \mathcal{K}$.
Set $\psi '' = \psi' - \mathcal{E} \mathcal{K} \mathcal{T} \psi'$, and observe that this spinor is smooth, vanishes on the ends of $M_+^n$, and lies within $\mathcal{H}$ since
\begin{align}
\mathcal{K}\mathcal{T} \psi'' = \mathcal{K}\mathcal{T}\psi' - \mathcal{K}^2 \mathcal{T}\psi' = 0.
\end{align} 
Furthermore, using that $\mathcal{E}\mathcal{K}\mathcal{T} \psi = 0$ and the boundedness of the maps involved yields
\begin{align}
\|\mathcal{E}\mathcal{K}\mathcal{T} \psi'\|_{W^{1,2}_{-q_*}(M')} &= \|\mathcal{E}\mathcal{K}\mathcal{T} (\psi' - \psi)\|_{W^{1,2}_{-q_*}(M')} \leq C \|\psi' - \psi\|_{W^{1,2}_{-q_*}(M')}
\end{align}
for some constant $C$, and thus the triangle inequality gives
\begin{align}
\|\psi - \psi'' \|_{W^{1,2}_{-q_*}(M')} \leq (1+C)\|\psi' - \psi\|_{W^{1,2}_{-q_*}(M')}.
\end{align}
The spinor $\psi''$ is the desired approximation to $\psi$.
\end{proof}

Continuing with the development of the technical tools needed for the proof of existence, we will next establish Poincar\'e-type estimates for spinors in the function space $\mathcal{H}$.

\begin{lemma}\label{SpaceTimeNablaEstimate}
Let $(M_\pm^n,g_\pm,k_\pm)$ be initial data sets with $M_-^n$ compact and $M_+^n$ asymptotically flat. 
Then there exists a constant $C$ depending on the geometry of the initial data such that:
\begin{enumerate}[label=(\roman*), topsep=3pt, itemsep=0ex]
\item $\|\psi\|^2_{W^{1,2}_{-q_*}(S_+)} \leq C \|\overline\nabla^+ \psi \|^2_{L^2(S_+)}$ for $\psi \in W^{1,2}_{-q_*}(S_+),$
\item $\|\psi\|^2_{W^{1,2}(S_-)} \leq C \left(\|\overline{\nabla}^- \psi\|^2_{L^2(S_-)} + \|\psi\|^2_{L^2(S_-|_{\Sigma^{n-1}})} \right)$ for $\psi \in W^{1,2}(S_-)$.
\end{enumerate}
\end{lemma}

\begin{proof}
Since $M_+^n$ is connected and contains an asymptotically flat end, the result \cite{BartnikChrusciel}*{Theorem 9.5} implies
that there exists $C_1$ such that the weighted Poincar\'e inequality
\begin{align} \label{STNablaPoincare}
    \|\psi/r\|^2_{L^2(S_+)} \leq C_1 \|\overline \nabla^+ \psi \|^2_{L^2(S_+)}
\end{align}
holds for all $\psi\in W^{1,2}_{-q_*}(S_+)$. Therefore, there exists a constant $C_2$ so that
\begin{align}\label{fainofinaqoihh}
\begin{split}
    \| \psi \|_{W^{1,2}_{-q_*}(S_+)} &\leq \|\nabla^+ \psi\|_{L^2(S_+)} + \|\psi/r\|_{L^2(S_+)} \\
    & \leq \|\overline{\nabla}^+ \psi \|_{L^2(S_+)} +  \||k_+|_{g_+} \psi \|_{L^2(S_+)}+ \|\psi/r\|_{L^2(S_+)}  \\
    &\leq \|\overline{\nabla}^+ \psi \|_{L^2(S_+)}  +  \left(\sup(r|k_+|) +1 \right) \|\psi / r\|_{L^2(S_+)} \\
    & \leq C_2 \|\overline{\nabla}^+ \psi \|_{L^2(S_+)},
\end{split}
\end{align}
where the last inequality follows from the asymptotic decay of $k_+$ and \eqref{STNablaPoincare}. This yields part $(i)$.

Consider now part $(ii)$. We will first demonstrate the following Poincar\'e-type inequality
\begin{align}\label{PoincareWBoundary}
     \|\psi\|^2_{L^2(S_-)} & \leq C_3 \left(\|\overline{\nabla}^- \psi\|^2_{L^2(S_-)} + \|\psi\|^2_{L^2(S_-|_{\Sigma^{n-1}})} \right),
\end{align}
for $\psi\in W^{1,2}(S_-)$. Suppose the inequality does not hold, then there exists a sequence of spinors $\psi_i$ such that $\|\psi_i\|_{L^2(S_-)} = 1$ and $\|\overline{\nabla}^- \psi_i\|^2_{L^2(S_-)} + \|\psi_i\|^2_{L^2(S_-|_{\Sigma^{n-1}})} \to 0$. As in \eqref{fainofinaqoihh}, this shows that the sequence is bounded in $W^{1,2}(S_-)$.
By Rellich's theorem, after passing to a subsequence, $\psi_i$ strongly converges in $L^2(S_-)$ and weakly converges in $W^{1,2}(S_-)$. Let $\psi = \lim_{i \to \infty} \psi_i$. Then by weak lower semi-continuity of the norm and the trace theorem, it follows that $\overline{\nabla}^- \psi = 0$ and $\psi|_{\Sigma^{n-1}} = 0$. 
In particular, the limit spinor satisfies
\begin{align}
    \nabla^-_i \psi = -\frac{1}{2}(k_-)^j_i e_j \tau_-\psi.
\end{align}
Hence $\psi$ is smooth, and
integrating along curves connecting interior points to the boundary where $\psi = 0$, we find that $\psi$ vanishes everywhere. This contradicts $\|\psi\|_{L^2(S_-)} = 1$, yielding the desired result.
\end{proof}

The two Poincar\'e inequalities of Lemma \ref{SpaceTimeNablaEstimate} may be pasted together using the boundary conditions, producing a global weighted Poincar\'e inequality for spinors in $\mathcal{H}$.

\begin{corollary}\label{t:mainPoincare}
    The inner product $\langle\langle f, g \rangle\rangle := \int_{M'} \left\langle \overline{\nabla} f , \overline{\nabla} g \right\rangle dV$ is equivalent to the $W^{1,2}_{-q_*}(M')$-inner product on $\mathcal{H}$. In particular, there is a constant $C$ such that 
\begin{equation}
    \|\psi\|_{W^{1,2}_{-q_*}(M')} \leq C\|\overline \nabla \psi\|_{L^2(M')}
\end{equation}
for all $\psi\in\mathcal{H}$.
\end{corollary}

\begin{proof}
Using that $(A+B\epsilon_+) \circ \Phi$ is a bounded invertible map from $L^2(S_-|_{\Sigma^{n-1}}) \to L^2(S_+|_{\Sigma^{n-1}})$, together with Lemma \ref{SpaceTimeNablaEstimate}, and continuity of the trace produces
\begin{align}
\begin{split}
    \|\psi\|^2_{W^{1,2}(S_-)} & \leq C \left( \|\overline\nabla^- \psi\|^2_{L^2(S_-)} +  \|\psi\|^2_{L^2(S_-|_{\Sigma^{n-1}})} \right) \\
    &  \leq C_1 \left( \|\overline\nabla^- \psi\|^2_{L^2(S_-)} +  \|\psi\|^2_{W^{1,2}_{-q_*}(S_+)} \right) \\
    &\leq C_2 \left( \|\overline\nabla^- \psi\|^2_{L^2(S_-)} +  \|\overline\nabla^+ \psi\|^2_{L^2(S_+)} \right),
\end{split}
\end{align}
for any $\psi \in \mathcal{H}$. Applying part $(i)$ of Lemma \ref{SpaceTimeNablaEstimate} again yields
\begin{align}
    \| \psi \|^2_{W^{1,2}_{-q_*}(M')} \leq C_3 \| \overline\nabla  \psi \|^2_{L^2(M')}.
\end{align}
The inequality $\| \overline\nabla  \psi \|^2_{L^2(M')}\leq C_4 \| \psi \|^2_{W^{1,2}_{-q_*}(M')}$ holds by the decay of $k$.
\end{proof}

To proceed, we establish a mass inequality for spinors that lie in $\mathcal{H}$ after subtracting a constant model spinor. 
This inequality will then be used to show that $D_W:\mathcal{H} \to L^2(M')$ is an isomorphism. 


\begin{proposition}\label{p:massFormula}
Let $\psi_0$ be a smooth spacetime spinor on $M'$ that is asymptotic to a constant spinor $\psi_\infty$. 
For any spinor $\psi$ satisfying $\psi - \psi_0 \in \mathcal{H}$, we have
    \begin{align}\begin{split}
        \label{eq:massFormulaEquation}
    \int_{M'}\left(|\overline\nabla \psi|^2 - |D_W \psi|^2 +\frac12\left\langle \psi,(\mu+J\tau)\psi \right\rangle \right)dV
    &\leq -\frac12 \int_{\Sigma^{n-1}} \mathcal{B} |\psi|^2 dA \\
    &+ \frac{(n-1)\omega_{n-1}}{2}  \left(E|\psi_\infty|^2-\left\langle\psi_\infty,P \tau\psi_\infty\right\rangle\right),
    \end{split}
\end{align}
where
\begin{align}
    \mathcal{B} = \left\langle  F(\vec{H}_-)-\vec{H}_+,  \nu_+ \right\rangle- \ \sqrt{ \left\langle  F(\vec{H}_-)-\vec{H}_+,\tau_+ \right\rangle^2  + |\beta^\Delta|_\gamma^2 } .
\end{align}
\end{proposition}

\begin{proof}
If $\psi - \psi_0 \in \mathcal{H} \cap C_c^\infty(M')$, then $\psi$ is a smooth spinor satisfying the boundary conditions and is asymptotic to $\psi_\infty$. By \eqref{eq:Lichnerowicz} as well as Propositions \ref{prop:boundarylimit} and \ref{p:BoundaryTerms}, the inequality holds for such $\psi$. The general inequality is established by approximating $\psi-\psi_0$ with smooth spinors. To carry this out, note that the first two terms of the left-hand side of \eqref{eq:massFormulaEquation} are continuous on $\{\mathcal{H} + \psi_0\}$ in the $W^{1,2}_{-q_*}(M')$-topology, due to the 
fall-off of $k$. For the third term on the left side, we have $\mu + J\tau = O(r^{-q-2})
= O(r^{-2})$, so this term is continuous as well. Moreover, the boundary term involving $\mathcal{B}$ is a continuous functional on $\{\mathcal{H} + \psi_0\}$ by the trace theorem. Thus, the density result 
Lemma \ref{l:smoothApproxH} implies that the mass inequality holds for all spinors in $\{\mathcal{H} + \psi_0\}$.
\end{proof}

\begin{theorem}\label{t:DiracWittenIso}
Assume that $(M_\pm^n,g_\pm,k_\pm)$ are as in Theorem \ref{theorem-main:PMT}. Then $D_W: \mathcal{H} \to L^2(M')$ is an isomorphism.
\end{theorem}

\begin{proof}
Since $k=O\left(|x|^{-q-1}\right)$, it follows that $D_W$ is a bounded linear operator from $\mathcal{H}$ to $L^2(M')$. We will first establish injectivity. Applying the mass formula of Proposition \ref{p:massFormula} to a spinor $\psi \in \mathcal{H}$ and noting that $\psi$ is asymptotic to $\psi_\infty=0$, the mass term vanishes and we have
\begin{align}
    \|\overline \nabla \psi \|_{L^2(M')}^2 - \|D_W\psi\|^2_{L^2(M')} \leq -\frac12 \int_{\Sigma^{n-1}} \mathcal{B} |\psi|^2 d A \leq 0,
\end{align}
by the DEC-creased condition and the dominant energy condition. Combining this with the weighted Poincar\'e inequality of Corollary \ref{t:mainPoincare}, we obtain
\begin{align}\label{eq:InjectivityEstimate}
\|\psi\|^2_{W^{1,2}_{-q_*}} \leq C \|\overline \nabla \psi \|_{L^2(M')}^2  \leq C \|D_W\psi\|^2_{L^2(M')},
\end{align}
from which injectivity follows.

Consider now surjectivity of $D_W$. Let $\eta \in L^2(M')$, then $\psi \mapsto \left\langle \eta,D_W  \psi\ \right\rangle_{L^2(M')}$ defines a bounded linear functional on $\mathcal{H}$. By \eqref{eq:InjectivityEstimate}, $\left\langle D_W\ \cdot\ ,  D_W \ \cdot\ \right\rangle_{L^2(M')}$ is equivalent to the inner product on $\mathcal{H}$. The Riesz representation theorem then provides the existence of a unique $\omega \in \mathcal{H}$ such that
\begin{align}
    \left\langle D_W \omega, D_W\psi \right\rangle_{L^2(M')} = \left\langle \eta ,D_W \psi \right\rangle_{L^2(M')}
\end{align}
for every $\psi\in\mathcal{H}$. Then setting $\varphi = D_W\omega - \eta \in L^2(M')$ produces
\begin{align}
    \left\langle \varphi, D_W \psi \right\rangle_{L^2(M')} = 0
\end{align}
for all $\psi \in \mathcal{H}$. In particular, $D_W \varphi = 0$ weakly, however we do not yet know its regularity at the boundary. Proposition \ref{BoundaryReg} below implies that $\varphi \in W^{1,2}_{loc}(M')$ and satisfies the boundary conditions that define $\mathcal{H}$.

To show $(-q_*)$-weighted Sobolev space decay of $\varphi$, we argue via approximation. Let $\chi_j$ be a nondecreasing sequence of cut-off functions such that for large $j$, $\chi_j \equiv 1$ inside the coordinate sphere $S_{2^j}$, $\chi_j = 0$ outside $S_{2^{j+1}}$, and $|\nabla \chi_k| \leq  2^{1-j}$. Observe that
\begin{align}
    D_W(\chi_j \varphi) = \chi_j D_W \varphi + (\nabla \chi_j) \varphi =  (\nabla \chi_j) \varphi,
\end{align}
where $(\nabla \chi_k) \varphi$ denotes Clifford multiplication of $\nabla \chi_j$ on $\varphi$. Furthermore, this and \eqref{eq:InjectivityEstimate} imply
\begin{align}
\begin{split}
    \|\chi_j \varphi - \chi_{j+1} \varphi\|_{W^{1,2}_{-q_*}(M')} &\leq C \|D_W(\chi_j \varphi - \chi_{j+1} \varphi)\|_{L^2(M')} \\
    &\leq C \|\nabla(\chi_k - \chi_{j+1}) \varphi\|_{L^2(M')} \\
    &\leq 2^{2-j} C  \|\varphi\|_{L^2(M')},
\end{split}
\end{align}
and therefore $\chi_j \varphi$ converges to $\varphi$ in $W^{1,2}_{-q_*}(M')$. Hence $\varphi$ lies in $\mathcal{H}$ and is a solution to $D_W \varphi = 0$. By the injectivity of $D_W$ we must have $\varphi = 0$, showing that $D_W \omega = \eta$.
This establishes surjectivity. Finally, by the bounded inverse theorem we find that the inverse is a bounded, thus establishing
the isomorphism property.
\end{proof}

We now arrive at the main existence statement and integral inequality.

\begin{proof}[Proof of Theorem \ref{theorem-main:PMT} and \ref{theorem-main:existence}]
Note that Theorem \ref{theorem-main:PMT} follows from Theorem \ref{theorem-main:existence}. To prove the latter result, fix a smooth spinor $\psi_0$ on $M'$ with $\psi_0 \equiv 0$ away from a neighborhood of the distinguished end $M^n_\ell$ and $\psi_0 = \psi_\infty$ on $M^n_\ell$. Since $\psi_\infty$ is constant,  $D_W \psi_0$ decays at the same order as $k_+$ and the connection coefficients of $g_+$, which is $O(|x|^{-q-1}) = O(|x|^{-n/2- \epsilon})$ for some $\epsilon>0$. So $D_W \psi_0 \in L^2(M') \cap C^\infty(M')$. By Theorem \ref{t:DiracWittenIso}, we obtain a solution $\omega \in \mathcal{H}$ to 
\begin{align}
    D_W \omega = - D_{W} \psi_0.
\end{align}
The regularity statement $\omega \in C^\infty(M')$ follows from Proposition \ref{p:HigherBoundaryReg} demonstrated in the next section. Then $\psi = \omega + \psi_0$ is the desired spinor. Now apply Proposition \ref{p:massFormula} to $\psi$ together with the DEC-creased condition to obtain the desired result. 
\end{proof}

\section{Ellipticity of the Boundary Conditions}

The goal of this section is to establish the regularity result Proposition \ref{p:HigherBoundaryReg} for weak solutions to \eqref{eq:mainPDE}, which is used in the proof of Theorem \ref{theorem-main:existence} above. Doing so requires us to understand \eqref{eq:mainPDE} as a type of elliptic boundary value problem. In \cite{BaerBallmann}, elliptic boundary conditions are defined for Dirac operators, which we will now briefly review. Let $S$ be the spacetime spinor bundle on a spin initial data set $(M^n,g,k)$ with compact boundary. Let $D_W: H^1(S) \to L^2(S)$ be the Dirac-Witten operator, and recall it is formally self-adjoint. Let $D_{\max}$ be the extension of $D_W$ to $\dom(D_{\max}) \subset L^2(M^n)$ defined by: $\varphi \in \dom(D_{\max})$ whenever there exists $\eta \in L^2(M^n)$ such that
\begin{align}
\left\langle \varphi, D_W \psi \right\rangle_{L^2(M^n)} = \left\langle \eta, \psi \right\rangle_{L^2(M^n)}
\end{align}
for all smooth spinors $\psi$ compactly supported in the interior of $M^n$. This equation indicates that $\varphi$ is an $L^2$-weak solution of $D_W \varphi = \eta$ with no boundary conditions imposed, and thus in this case we define $D_{\max} \varphi := \eta$.  Note that $\dom(D_{\max})$ is complete with the graph norm
\begin{align}
    \|\varphi \|_{\dom(D_{\max})}  = \| \varphi\|_{L^2(M^n)} + \|D_{\max} \varphi\|_{L^2(M^n)},
\end{align} 
and standard elliptic theory implies that $\varphi$ is $H^1$ in the interior. A first-order operator $A$ on $S|_{\partial M^n}$ is an \textit{adapted operator} if the principal symbol $\sigma_A$ of $A$ satisfies
\begin{align}
    \sigma_A(\xi, x) = \nu  \sigma_{D_W}(\xi, x)
\end{align}
for all $x\in \partial M^n$, $\xi \in T^*_x(\partial M^n)$, where $\nu$ is Clifford multiplication by the outward unit normal to $\partial M^n$. In particular, the boundary Dirac operator $\mathcal{D}^{\partial M^n} = \nu e^\alpha \nabla^{\partial M^n}_\alpha$ is an adapted operator.

Now let us return to the setting where $M' = M_+^n \sqcup M_-^n$. 
Let $H^s_{\geq 0}(\mathcal{D}^{\partial M'})$ (resp. $H^s_{< 0}(\mathcal{D}^{\partial M'})$) be subspaces of $H^s(S|_{\partial M'})$ spanned by the eigenspaces of the boundary Dirac operator $\mathcal{D}^{\partial M'}$ of nonnegative (resp. negative) eigenvalues. Define the hybrid Sobolev space
\begin{align}
    \check{H}(\mathcal{D}^{\partial M'}) := H^{\frac{1}{2}}_{< 0}(\mathcal{D}^{\partial M'})\oplus  H^{-\frac{1}{2}}_{\geq 0}(\mathcal{D}^{\partial M'}).
\end{align}
As $M'$ is a complete Riemannian manifold with boundary, the result \cite{BaerBallmann}*{Theorem 6.7} asserts that the trace map uniquely extends to a surjective bounded linear map $\mathcal{T}: \dom(D_{\max}) \to \check{H}(\mathcal{D}^{\partial M'})$, that $\varphi \in \dom(D_{\max})\cap H^1_{loc}(M')$ if and only if both $\varphi\in \dom(D_{\max})$ and $\mathcal{T}(\varphi) \in H^{\frac{1}{2}}(\partial M')$, and the integration by parts formula 
\begin{align}\label{weakIntByParts}
    \left\langle D_{\max} \varphi, \psi \right\rangle_{L^2(M')} - \left\langle \varphi, D_{\max} \psi\right\rangle_{L^2(M')} = \left\langle \nu \mathcal{T} \varphi, \mathcal{T} \psi \right\rangle_{L^2(\partial M')}
\end{align}
holds for $\varphi,\psi\in \dom(D_{\max})$. The paring on the right side of \eqref{weakIntByParts} is well-defined since Clifford multiplication by $\nu$ swaps the positive and negative eigenspaces from $\check{H}(\mathcal{D}^{\partial M'})$. We begin with a preliminary regularity result for weak solutions to the boundary value problem.

\begin{proposition}\label{BoundaryReg}
If $\varphi, \eta \in L^2(M')$ are such that
\begin{align}
    \left\langle \varphi, D_W \psi\right\rangle_{L^2(M')} = \left\langle \eta, \psi  \right\rangle_{L^2(M')}\quad\quad \text{ for all }\psi \in \mathcal{H},
\end{align}
then $\varphi \in H^1_{loc}(M')$ and $\varphi$ satisfies the boundary conditions defining $\mathcal{H}$.
\end{proposition}

\begin{proof}
Since $\mathcal{H}$ contains smooth spinors compactly supported on the interior of $M'$, we have $D_{\max} \varphi = \eta$. Using \eqref{weakIntByParts}, for all $\psi \in \mathcal{H}$ it holds that
\begin{align}
\begin{split}
    0 &= \left\langle \nu \mathcal{T} \varphi, \mathcal{T} \psi \right\rangle_{L^2(\partial M')} \\
    &= \left\langle \nu_- \mathcal{T} \varphi, \mathcal{T} \psi \right\rangle_{L^2(\partial M_-^n)}  - \left\langle \nu_+ \mathcal{T} \varphi, \mathcal{T} \psi \right\rangle_{L^2(\partial M_+^n)} \\
    &= \left\langle \nu_+ \Phi(\mathcal{T} \varphi), \Phi(\mathcal{T} \psi )\right\rangle_{L^2(\partial M_+^n)}  - \left\langle \nu_+ \mathcal{T} \varphi, \mathcal{T} \psi \right\rangle_{L^2(\partial M_+^n)} \\
    &= \left\langle \nu_+ \Phi(\mathcal{T} \varphi),(A+B\epsilon_+ )  \mathcal{T} \psi\right\rangle_{L^2(\partial M_+^n)}  - \left\langle \nu_+ \mathcal{T} \varphi, \mathcal{T} \psi \right\rangle_{L^2(\partial M_+^n)} \\
    &= \left\langle \nu_+ [(A-B\epsilon_+) \Phi(\mathcal{T} \varphi) - \mathcal{T} \varphi], \mathcal{T} \psi)\right\rangle_{L^2(\partial M_+^n)}. 
\end{split}
\end{align}
Since $\mathcal{T} \psi$ is an arbitrary function in $H^{\frac{1}{2}}(\partial M_+^n)$, it follows that
\begin{align}\label{eq:boundaryConditionsWeak}
    (A-B\epsilon_+) \Phi(\mathcal{T} \varphi) = \mathcal{T} \varphi \quad\text{ on }\partial M_+^n
\end{align}
in the sense of distributions, justifying the second claim of the proposition.
    
We know that $\mathcal{T} \varphi$ is in $\check{H}(\mathcal{D}^{\partial M'})$ by \cite{BaerBallmann}*{Theorem 6.7 $(ii)$}. It remains to show that $\mathcal{T}\varphi$ fully lies within $H^{\frac{1}{2}}(\partial M')$, which by virtue of $(iii)$ from the same theorem will imply that $\varphi \in H^{1}_{loc}(M')$.
To do this, we revisit the proof of \cite{BaerBallmann}*{Theorem 7.20} that equates certain pseudolocal boundary conditions with the well-studied notion of an elliptic boundary value problem. Although the current boundary conditions \eqref{eq:mainPDE} are a form of transmission conditions, making them neither local nor pseudolocal, they can be recast as pointwise boundary conditions in an axillary bundle over the boundary. The space of sections on the boundary is $H^{\frac{1}{2}}(\partial M') =  H^{\frac{1}{2}}(S_+|_{\Sigma^{n-1}}\oplus S_-|_{\Sigma^{n-1}})$, and we define an isometry with $H^{\frac{1}{2}}_{+}:=H^{\frac{1}{2}}(S_+|_{\Sigma^{n-1}} \oplus S_+|_{\Sigma^{n-1}})$ via
\begin{align}
    \varphi_+\oplus \varphi_- \mapsto \varphi_+\oplus \Phi(\varphi_-).
\end{align}
The boundary Dirac operator on $ H^{\frac{1}{2}}(S_+|_{\Sigma^{n-1}}) \oplus H^{\frac{1}{2}}(S_-|_{\Sigma^{n-1}})$ is given by
\begin{align}
    \mathcal{D}^{\partial M'} = \mathcal{D}_+\oplus \mathcal{D}_-.
\end{align}
Moreover, \eqref{eq:PhiCommuteDirac} yields
\begin{align}
    \mathcal{D_+}\varphi_+\oplus \Phi(\mathcal{D_-}\varphi_-) = \mathcal{D_+}\varphi_+\oplus -\mathcal{D_+}\Phi(\varphi_-),
\end{align}
which implies that $\mathcal{D}^{\partial M'}$ acts as $\mathcal{D}_+\oplus - \mathcal{D}_+$ on $H^{\frac{1}{2}}_{+}$.
    
In what follows we will work on $H^{\frac{1}{2}}_{+}$. Define $\mathcal{K}: H^{\frac{1}{2}}_{+} \to H^{\frac{1}{2}}_{+}$ via
\begin{align}
    \mathcal{K}(\psi_1\oplus \psi_2) = \frac{1}{2}\left(\psi_1 - (A-B\epsilon_+ )\psi_2\right)\oplus \frac12 \left(\psi_2 - (A+B\epsilon_+) \psi_1\right),
\end{align}
whose kernel defines the boundary conditions. Let $Q_{<0}$ be the $L^2$-projection onto the negative eigenspace of $\mathcal{D}_+\oplus -\mathcal{D}_+$. We claim that $\mathcal{K}-Q_{<0}$ is an elliptic pseudo-differential operator of order $0$. It is known that the principal symbol $\sigma_{Q_{<0}}(\xi)$ of $Q_{<0}$ is the orthogonal projection onto the negative eigenspace of $i(\sigma_{\mathcal{D}_+}(\xi)\oplus \sigma_{-\mathcal{D}_+}(\xi))= i (-\nu_+ \xi\oplus  \nu_+ \xi)$; for details see the proof of \cite{BaerBallmann}*{Theorem 7.20}. Without loss of generality we may assume that $|\xi| = 1$.
Then the symmetric operator $i \nu_+ \xi$ squares to 1, so it has eigenvalues $\pm 1$, and hence
\begin{align}
\sigma_{Q_{<0}}(\xi) = \frac{1}{2}\left(1 + i\nu_+ \xi\right)\oplus\frac12 \left( 1 - i\nu_+\xi\right).
\end{align}
We now check that the principal symbol $\sigma_\mathcal{K}(\xi) - \sigma_{Q_{<0}}(\xi)$ is injective, and consequently an isomorphism. Suppose that $(\sigma_\mathcal{K}(\xi) - \sigma_{Q_{<0}}(\xi))(\psi_1\oplus \psi_2) = 0$, and observe that this implies
\begin{equation}\label{foanofinahg}
    (A - B\epsilon_+)\psi_2 = -i \nu_+ \xi \psi_1 , \quad\quad\quad
    (A + B\epsilon_+)\psi_1 =  i \nu_+ \xi \psi_2.
\end{equation}
Solving for $\psi_1$ in the second equation and inserting it into the first yields
\begin{align}
    (A - B\epsilon_+ ) \psi_2 = -i\nu_+ \xi (A - B\epsilon_+ ) i\nu_+ \xi\psi_2.
\end{align}
It follows that $(A-B\epsilon_+)\psi_2=0$. Therefore $\psi_2=0$, and \eqref{foanofinahg} then gives $\psi_1=0$.
This verifies that $\sigma_{\mathcal{K}-Q_{_{<0}}}(\xi)$ is an isomorphism for $\xi \neq 0$, so that $\mathcal{K}-Q_{<0}$ is an elliptic pseudo-differential operator. In particular, there is a zeroth order parametrix $R$ 
such that $R(\mathcal{K}-Q_{<0}) = I + \mathcal{S}$ where $\mathcal{S}$ is a smoothing operator. Due to the boundary condition $\mathcal{K}\mathcal{T}\varphi = 0$, we have
\begin{align}\label{eq:Parametrix}
     \mathcal{T}\varphi + \mathcal{S}\mathcal{T}\varphi = R(\mathcal{K} - Q_{<0}) \mathcal{T}\varphi = R Q_{<0} \mathcal{T}\varphi.
\end{align}
Furthermore, observe that $\mathcal{T}\varphi \in \check{H}(\mathcal{D}^{\partial M'})$ implies $Q_{<0} \mathcal{T}\varphi \in H^{\frac{1}{2}}(\partial M')$, and therefore $RQ_{<0}\mathcal{T}\varphi \in H^{\frac{1}{2}}(\partial M')$. Since $\mathcal{S}$ is a smoothing operator $\mathcal{S}\mathcal{T}\varphi \in C^\infty(\partial M')$, thus \eqref{eq:Parametrix} implies the desired outcome that $\mathcal{T}\varphi \in H^{\frac{1}{2}}(\partial M')$.
\end{proof}

The following proposition describes the required higher-order regularity of solutions to the boundary value problem.

\begin{proposition}\label{p:HigherBoundaryReg}
Assume that $(M_\pm,g_\pm,k_\pm)$ are as in Theorem \ref{theorem-main:PMT}. If $\eta \in H^s_{loc}(M')$ for $s\geq0$ and $\varphi \in \mathcal{H}$ satisfies 
\begin{align}
D_W \varphi = \eta,
\end{align}
then $\varphi \in H^{s+1}_{loc}(M')$.
\end{proposition}

\begin{proof}
The result is a consequence of \cite{BaerBallmann}*{Theorem 7.17}, provided that $\mathcal{K}$ satisfies the so-called $(s+1/2)$-regular property \cite{BaerBallmann}*{Definition 7.15} for all $s\geq0$. 
This property was demonstrated in the the proof of Proposition \ref{BoundaryReg} for $s=0$, and the proof is identical for $s>0$.
\end{proof}

\begin{remark}
We note that the spinor PDE boundary value problem can be shown to satisfy the ellipticity conditions of
Shapiro-Lopatinski (or Agmon-Douglis-Nirenberg). 
This gives an alternative route to proving the above regularity results.
\end{remark}

\section{Rigidity}


To investigate the case of $m=0$ and establish Theorem \ref{thm:rigidity}, we adopt the strategy of Beig-Chru\'{s}ciel \cite{BeigChrusciel}. First, some preliminary notions and facts are needed. Given an initial data set $(M^n,g,k)$, a {\emph{lapse-shift}} pair $(u,Y)$ consists of a function $u$ and a vector field $Y$ on $M^n$.

\begin{definition}
Given a lapse-shift pair $(u,Y)$ on $M^n$ with $u>0$, the {\emph{Killing development associated with $(u,Y)$}} is the Lorentzian manifold $(\mathbb{R}\times M^n,\bar g)$ where
\begin{align}\label{eq:killingdev}\begin{split}
    \bar g&=-u^2dt^2+g_{ij}(dx^i+Y^idt)(dx^j+Y^jdt)\\
    {}&=-(u^2-|Y|^2)dt^2 +2Y dt+g.
    \end{split}
\end{align}
\end{definition}

Since $(u,Y)$ are independent of $t$, the new vector field $\partial_t$ is Killing for $\bar{g}$. Moreover, $(M^n,g)$ isometrically embeds into the Killing development as $\{t=0\}\times M^n$. If $\tau$ denotes the unit timelike normal to this embedding, notice that $\partial_t$ decomposes into the orthogonal sum $u\tau+Y$. The following describes when the second fundamental form of $\{t=0\}$ agrees with $k$.

\begin{proposition}{\cite{BeigChrusciel}*{Section 2}}\label{prop:killingSSF}
Suppose $(M^n,g,k)$ is an initial data set with a lapse-shift pair $(u,Y)$ satisfying $u>0$. If
\begin{equation}
    \mathcal{L}_{Y}g +2uk=0,\quad\quad\quad d u+k(Y,\cdot)=0,
\end{equation}
where $\mathcal{L}$ denotes Lie differentiation, then $(M^n,g,k)$ isometrically embeds into the Killing development associated with $(u,Y)$ as $\{t=0\}$ with induced second fundamental form $k$.
\end{proposition}


To implement the Killing development construction, we need a source of lapse-shift pairs $(u,Y)$. The next proposition produces such pairs from spacetime spinors, and describes the algebraic properties induced by the boundary conditions \eqref{eq:mainPDE}.

\begin{proposition}\label{prop:LScomponents}
Suppose that $(M_\pm^n,g_\pm,k_\pm)$ is a DEC-creased initial data set such that $M^n$ is spin. Let $\psi_\pm$ be spacetime spinors on $M_\pm^n$ satisfying the boundary condition $\Phi(\psi_-)=(A+B\epsilon_+)\psi_+$ on $\Sigma^{n-1}$. Consider the lapse-shift pair $(u_\pm,Y_\pm)$ defined by 
\begin{align}\label{eq:LScomponents}
   u_\pm=|\psi_\pm|^2,\quad\quad \langle Y_\pm,W\rangle=\langle \tau_\pm W\psi_\pm,\psi_\pm\rangle,\quad\quad W\in TM_\pm^n.
\end{align}
Then $Y_\pm$ is real and the following equations along $\Sigma^{n-1}$ hold
\begin{align}\begin{split}\label{eq:LorentzRelation}
    \langle Y_-,V\rangle&=\langle Y_+,V\rangle,\quad\quad V\in T\Sigma^{n-1},\\
    \langle Y_-,\nu_-\rangle&=a\langle Y_+,\nu_+\rangle-bu_+,\quad\quad u_-=au_+-b\langle Y_+,\nu_+\rangle.
    \end{split}
\end{align}
Moreover if $\overline{\nabla}\psi_\pm=0$, then $\nabla Y_{\pm}$ is a symmetric tensor and the corresponding lapse-shift pair $(u_\pm,Y_\pm)$ defined by \eqref{eq:LScomponents} satisfies
    \begin{align}\label{eq:parallelLS}
        \mathcal{L}_{Y_\pm} g_\pm+2u_\pm k_\pm=0,\quad\quad\quad d u_\pm+k_\pm(Y_\pm,\cdot)=0.
    \end{align}
\end{proposition}

\begin{proof}
To see that $Y_\pm$ is real, simply note that
\begin{equation}
    \langle \tau_\pm W\psi_\pm,\psi_\pm\rangle=\langle \psi_\pm,\tau_\pm W\psi_\pm\rangle=\overline{\langle \tau_\pm W\psi_\pm,\psi_\pm\rangle}.
\end{equation}
For $V\in T\Sigma^{n-1}$ we have
\begin{align}\begin{split}
    \left\langle Y_-,V\right\rangle&=\left\langle \tau_-V\psi_-,\psi_-\right\rangle\\
    &=\left\langle \tau_+V(A+B\epsilon_+ )\psi_+,(A+B\epsilon_+ )\psi_+\right\rangle\\
    {}&=\left\langle \tau_+(A+B\epsilon_+)V\psi_+,(A+B\epsilon_+ )\psi_+\right\rangle\\
    {}&=\left\langle (A-B\epsilon_+ )\tau_+V\psi_+,(A+B\epsilon_+ )\psi_+\right\rangle=\left\langle \tau_+V\psi_+,\psi_+\right\rangle=\left\langle Y_+,V\right\rangle.
    \end{split}
\end{align}
To verify the remaining identities we compute
\begin{align}
\begin{split}
    \left\langle Y_-,\nu_-\right\rangle&=\left\langle \tau_-\nu_-\psi_-,\psi_-\right\rangle\\
    &=\left\langle \tau_+\nu_+(A+B\epsilon_+ )\psi_+,(A+B\epsilon_+ )\psi_+\right\rangle\\
    {}&=\left\langle \tau_+(A-B\epsilon_+ )\nu_+\psi_+,(A+B\epsilon_+ )\psi_+\right\rangle\\
    {}&=\left\langle (A+B\epsilon_+)\tau_+(A-B\epsilon_+ )\nu_+\psi_+,\psi_+\right\rangle\\
    {}&=\left\langle \tau_+(A-B\epsilon_+)^2\nu_+\psi_+,\psi_+\right\rangle\\
    {}&=\left\langle \tau_+(a-b\epsilon_+)\nu_+\psi_+,\psi_+\right\rangle\\
    {}&=a \left\langle \tau_+\nu_+\psi_+,\psi_+\right\rangle - b \left\langle \tau_+ \epsilon_+ \nu_+\psi_+,\psi_+\right\rangle \\
    {}&=a \left\langle \tau_+\nu_+\psi_+,\psi_+\right\rangle - b \left\langle \psi_+,\psi_+\right\rangle =a \left\langle Y_+,\nu_+\right\rangle - b u_+,
\end{split}
\end{align}
and 
\begin{align}
\begin{split}
    u_-=\langle\Phi(\psi_-),\Phi(\psi_-)\rangle
    {}&=\left\langle(A+B\epsilon_+ )\psi_+,(A+B\epsilon_+ )\psi_+\right\rangle\\
    {}&=\left\langle (A+B\epsilon_+)^2\psi_+,\psi_+\right\rangle\\
    {}&=\left\langle (a+b\epsilon_+ )\psi_+,\psi_+\right\rangle\\
    {}&=a \left\langle \psi_+,\psi_+\right\rangle + b \left\langle \epsilon_+ \psi_+,\psi_+\right\rangle =au_+ -b \langle Y_+ ,\nu_+ \rangle .
\end{split}
\end{align}
To see that $\nabla Y_\pm$ is symmetric when $\overline\nabla \psi_\pm=0$ observe that by differentiating the second equation of \eqref{eq:LScomponents} (foregoing the $\pm$ notation) produces
\begin{align}
\begin{split}
\langle\nabla_{X} Y,W\rangle &=\langle\tau W\nabla_{X}\psi,\psi\rangle+\langle\tau W\psi,\nabla_{X}\psi\rangle\\
&=-\frac{1}{2}\left(\langle\tau W k(X,\cdot) \tau \psi,\psi\rangle+\langle\tau W\psi,k(X,\cdot)\tau\psi\rangle\right)\\
&=-\frac{1}{2}\langle(Wk(X,\cdot)+k(X,\cdot)W)\psi,\psi\rangle=-k(X,W)|\psi|^2.
\end{split}
\end{align}
The final assertion in the proposition is contained in the proof of \cite{ChruscielMaerten}*{Theorem 3.1}.
\end{proof}


Before proceeding to the proof of the main rigidity result, we describe here a construction that will be used frequently. Let $(M_\pm^n,g_\pm,k_\pm)$ be an asymptotically flat DEC-creased initial data set. Given a vector $v\in\mathbb{R}^n$, take a constant spacetime spinor $\psi_0$ on $\mathbb{R}^n$ such that $v\cdot W=\langle \tau W\psi_0,\psi_0\rangle$ for all $W\in\mathbb{R}^n$. Let $\psi_0$ then be the constant spinor at infinity for a designated end $M^n_{\ell_1}$ in boundary value problem \eqref{eq:mainPDE}. The solution of this problem provided by Theorem \ref{theorem-main:existence} will be denoted by $\psi_\pm[v]$, and the corresponding lapse-shift pair defined by \eqref{eq:LScomponents} will be denoted $(u_\pm[v],Y_\pm[v])$.

\begin{proof}[Proof of Theorem \ref{thm:rigidity}]
The first step is to show that $E = |{P}|$ implies that $E=|P|=0$. Suppose $P\neq0$. In the notation above, we find solutions $\psi_\pm^P:=\psi_\pm[-P/|P|]$ to \eqref{eq:mainPDE} and apply Theorem \ref{theorem-main:existence} to find
\begin{align}\label{eq:massis0}
    0&=\int_{M_-^n}\!\!\left( | \overline{\nabla} \psi_-^P|^2 + \frac{1}{2} \left\langle \psi_-^P, (\mu + J\tau_-) \psi_-^P \right\rangle\right) dV +\int_{M_+^n}\!\!\left( | \overline{\nabla} \psi_+^P|^2 + \frac{1}{2} \left\langle \psi_+^P, (\mu + J\tau_+) \psi_+^P \right\rangle \right) dV.
\end{align}
Using the dominant energy condition, we immediately conclude that $|\overline\nabla\psi_\pm^P |\equiv0$. At this point we may apply the proof of \cite{ChruscielMaerten}*{Theorem 3.2} to obtain $E=|P|=0$. Strictly speaking, \cite{ChruscielMaerten} works with a smooth complete initial data set, however the computation \cite{ChruscielMaerten}*{Theorem 2.5} of $E=|P|=0$ only relies on the existence of a nontrivial $\overline\nabla$-parallel spinor in the asymptotic end.
As a consequence, the mass formula Theorem \ref{theorem-main:existence} implies that {\emph{any}} solution to \eqref{eq:mainPDE} with $\psi_0\neq0$ is parallel with respect to $\overline{\nabla}$.

Next, we construct an asymptotically time-like lapse-shift pair. Solving \eqref{eq:mainPDE} three times yields the following combination of lapse-shift pairs:
\begin{align}
\begin{split}
    Y_0^\pm&=-Y_\pm[\tfrac12(e_1+e_2)]+Y_\pm[\tfrac12(-e_1+e_2)]+Y_\pm[e_1],\\
    u_0^\pm&=-u_\pm[\tfrac12(e_1+e_2)]+u_\pm[\tfrac12(-e_1+e_2)]+u_\pm[e_1],
\end{split}
\end{align}
where $\{e_i\}_{i=1}^n$ denote the standard basis of $\mathbb{R}^n$. Note that $Y_0^+\to0$ and $u_0^+\to1$ as $|x|\to\infty$. Since the underlying spinors are $\overline{\nabla}$-parallel, we may apply the final statement of Proposition \ref{prop:LScomponents} to find
\begin{equation}\label{6.12}
\mathcal{L}_{Y_0^\pm} g_\pm+2u_0^\pm k_\pm=0,\quad\quad\quad d u_0^\pm+k_\pm(Y_0^\pm,\cdot)=0.
\end{equation}

We claim that $u^\pm_0>0$. To see this, first consider a curve $\sigma:[0,s_0]\to M_+^n$ connecting an arbitrary point $\sigma(0)\in M_+^n$ to a point $\sigma(T)$ far out in the end so that $u_0^+(\sigma(T))>0$ from the asymptotics. Observe that \eqref{6.12} and the symmetry of $\nabla Y^{+}_0$ imply 
\begin{equation}\label{eq:timelikesystem}
    \frac{d}{ds}((u^+_0)^2-|Y^+_0|^2)=2(-u^+_0k_+(Y^+_0,\dot\sigma)+u^+_0k_+(Y^+_0,\dot\sigma))=0.
\end{equation}
It follows that the squared Lorentz length $-(u^+_0)^2+|Y^+_0|^2$ remains negative throughout $M_+^n$ and so $u^+_0>0$. In particular, $-(u^+_0)^2+|Y^+_0|^2<0$ holds on $\Sigma^{n-1}$. 
According to Proposition \ref{prop:LScomponents}, the pairs $(u^\pm_0,Y^\pm_0)$ are related by an $SO^+(n,1)$-transformation along $\Sigma^{n-1}$, and we may conclude that $-(u^-_0)^2+|Y^-_0|^2<0$ is also satisfied along $\Sigma^{n-1}$. The above ODE argument may then be applied on $M_-^n$ to conclude that $u^-_0>0$, proving the claim.

The Killing developments of $(M_\pm^n,g_\pm,k_\pm)$ associated to $(u^\pm_0,Y^\pm_0)$ have a common time-like boundary $\mathbb{R}\times \Sigma^{n-1}$, and we denote the union of these developments along this boundary by $(N^{n+1},\overline g)$ with Killing vector $\partial_t$. Evidently, $\overline{g}$ is smooth away from $\mathbb{R}\times\Sigma^{n-1}$, and Proposition \ref{prop:killingSSF} implies that the original initial data sets $(M_\pm^n,g_\pm,k_\pm)$ embed isometrically as $\{t=0\}\times M_\pm^n$ with induced second fundamental forms $k_\pm$. We further claim that $\overline{g}$ is continuous across $\mathbb{R}\times\Sigma^{n-1}$, which is to say that the metrics
\begin{equation}
    -((u^\pm_0)^2-|Y_0^\pm|^2)dt^2+2Y^\pm_0 |_{\Sigma^{n-1}} dt+g_{\Sigma^{n-1}}
\end{equation}
are equal along $\Sigma^{n-1}$. Indeed, Proposition \ref{prop:LScomponents} implies that $Y^+_0 |_{\Sigma^{n-1}}=Y^0_- |_{\Sigma^{n-1}}$. Furthermore, as discussed above the squared Lorentz lengths $-(u^\pm_0)^2 +|Y_0^\pm|^2$ 
also agree at the crease.

To show that $\overline{g}$ is flat where it is smooth, a collection of asymptotically linearly independent lapse-shift pairs is needed. Consider the collection of pairs $\{(u_\pm[e_i],Y_\pm[e_i])\}_{i=1}^n$, which we extend to all of $N^{n+1}$ in a $t$-independent manner.
Let $\pmb{\tau}_\pm$ denote the unit timelike normal to the constant time slices $\{t\}\times M^n_\pm$ in $N^{n+1}$,
and consider the spacetime vector fields 
\begin{equation}
    X_0^\pm=u^\pm_0\pmb{\tau}_\pm+Y_0^\pm,\quad\quad X^\pm_{i}=(u_\pm[e_i]-u^\pm_0)\pmb{\tau}_\pm+Y_\pm[e_i]-Y_0^\pm,\quad i=1,\dots, n.
\end{equation}
Notice that $X_a^+$ converges to $e_a$ in the asymptotic end for $a=0,\dots,n$. Well-known calculations \cite{BeigChrusciel}*{(4.22)} show that $\nabla^{N^{n+1}}X^\pm_{a}=0$. Since $\{X_a^+\}_{a=0}^n$ are orthonormal at infinity and $X^-_a$ and $X^+_a$ are related by a Lorentz transformation along $\mathbb{R}\times \Sigma^{n-1}$, they remain orthonormal throughout $N^{n+1}$. We conclude that $(N^{n+1},\overline g)$ is flat away from $\mathbb{R}\times\Sigma^{n-1}$.

We will now show that $N^{n+1}$ is Lipschitz homeomorphic to $\mathbb{R}^{1,n}$, via a local isometry away from $\mathbb{R}\times \Sigma^{n-1}$. We first claim that the components of $X_a^\pm$ tangent to $\mathbb{R}\times\Sigma^{n-1}$ agree. To see this, notice that $\pmb{\tau}_\pm=(u_0^{\pm})^{-1}(\partial_t-Y_0^\pm)$ and compute
\begin{align}\begin{split}
    \langle u_\pm[e_i] \pmb{\tau}_\pm + Y_\pm[e_i], \partial_t\rangle &= - u_\pm[e_i] u_0^\pm + \langle Y_\pm[e_i], Y_0^\pm \rangle, \\
     \langle u_\pm[e_i] \pmb{\tau}_\pm + Y_\pm[e_i], V\rangle &= \langle Y_\pm[e_i], V\rangle,\quad \quad V \in T\Sigma^{n-1}.
\end{split}
\end{align}
It follows that the Lorentz invariance of Proposition \ref{prop:LScomponents} provided by \eqref{eq:LorentzRelation} implies 
\begin{equation}\label{eq:Xiagree}
\alpha_a^+|_{\mathbb{R}\times\Sigma^{n-1}}=\alpha_a^{-}|_{\mathbb{R}\times\Sigma^{n-1}},\quad\quad\quad a=0,\dots,n,
\end{equation}
where $\alpha^{\pm}_a$ are the dual 1-forms to the vector fields $X^{\pm}_a$. Furthermore, if $\overline\nu_\pm$ denote the spacelike normals to $\mathbb{R}\times\Sigma^{n-1}$ pointing into $\mathbb{R}\times M^n_+$, we have
\begin{align}\label{eq:Xanormal}\begin{split}
    \langle u_\pm[e_i] \pmb{\tau}_\pm + Y_\pm[e_i], \overline{\nu}_{\pm}\rangle &=  \langle u_\pm[e_i] \pmb{\tau}_\pm + Y_\pm[e_i], u_0^\pm \nu_\pm + \langle Y_0^\pm, \nu_\pm \rangle \pmb{\tau}_\pm \rangle \\
    &=   - u_\pm[e_i] \langle Y_0^\pm, \nu_\pm \rangle +  u_0^\pm \langle Y_\pm[e_i], \nu_\pm \rangle.
    \end{split}
\end{align}
Therefore, with the help of \eqref{eq:LorentzRelation} and $a^2-b^2=1$ it follows that
\begin{equation}\label{eq:alphanus}
\alpha_a^+(\overline{\nu}_+)=\alpha_a^-(\overline{\nu}_-).
\end{equation}

Next, pass to the universal cover $\widetilde{M}^n$ of $M^n$, and lift the data $g_\pm$, $k_\pm$, and $(u^\pm_0,Y^\pm_0)$. We will work on the Killing development $\widetilde{N}^{n+1}$ of $\widetilde{M}^n$, which may be identified with the universal cover of $N^{n+1}$. Denote the Levi-Civita connection of $\widetilde N^{n+1}$ by $\widetilde\nabla$ and write $\widehat{\Sigma}^{n}$ for the preimage of $\mathbb{R}\times \Sigma^{n-1}$. Let $\widetilde{N}^{n+1}_\pm$ be the preimage of the Killing developments of $M^n_\pm$ and consider the corresponding pullback 1-forms $\widetilde\alpha^{\pm}_a$, which are closed since they are $\widetilde{\nabla}$-parallel. Fix an anchor point $p\in \widetilde{N}^{n+1}$ and define functions $v_a(x)$ as the integral of $\widetilde\alpha_a^-\cup\widetilde\alpha_a^+$ along a path $\gamma$ joining $x$ to $p$. To see that this is independent of $\gamma$, consider two such choices $\gamma_1$ and $\gamma_2$. Then there is a  disc $D\subset\widetilde{N}^{n+1}$ bounded by the concatenation $\gamma_1*(-\gamma_2)$, which may be assumed to meet $\widehat{\Sigma}^{n}$ transversely. Letting $D_\pm=D\cap \widetilde{N}_\pm^{n+1}$ and $\sigma=D\cap\widehat{\Sigma}^{n}$, we may apply Stoke's theorem to find
\begin{align}
\begin{split}
    \int_{\gamma_1}(\widetilde\alpha_a^-\cup\widetilde\alpha_a^+)-\int_{\gamma_2}(\widetilde\alpha_a^-\cup\widetilde\alpha_a^+)&=\int_Dd(\widetilde\alpha_a^-\cup\widetilde\alpha_a^+)+\int_\sigma \widetilde\alpha_a^- -\int_\sigma \widetilde\alpha_a^+\\
    {}&=\int_{D_-}d\widetilde\alpha_a^-+\int_{D_+}d\widetilde\alpha_a^+\\
    {}&=0,
\end{split}
\end{align}
where in the penultimate line we used \eqref{eq:Xiagree}. This shows that $v_a$ is well-defined on $\widetilde{N}^{n+1}$. We note that these functions are globally Lipschitz, and smooth away from $\widehat{\Sigma}^{n}$. 

Let $\mathcal{V}:\widetilde{N}^{n+1}\rightarrow\mathbb{R}^{1,n}$ be given by $\mathcal{V}(x)=(v_0(x),\dots,v_n(x))$, and note this map is globally Lipschitz. The map is also smooth away from $\widehat{\Sigma}^{n}$, and at such points it gives a local diffeomorphism since the collection of 1-forms $dv_a=\widetilde{\alpha}^{\pm}_a$ are orthonormal. We claim that $\mathcal{V}$ is also a local homeomorphism near $\widehat\Sigma^n$. To see this, approximate $\mathcal{V}$ with its linearizations on either side of $\widehat\Sigma^n$ and note that this piecewise-linear map is bijective due to \eqref{eq:Xiagree} and \eqref{eq:alphanus}. Next consider a level set $\mathcal{M}^n$ of $v_0$, and note that it is globally Lipschitz and smooth away from $\widehat{\Sigma}^{n}$, while having a transverse intersection with $\widehat{\Sigma}^{n+1}$ since 
\begin{equation}
\langle\partial_t,\widetilde{\nabla}v_0\rangle=\langle\partial_t,u_{0}^{\pm}\pmb{\tau}_{\pm}+Y_{0}^{\pm}\rangle=-(u_{0}^{\pm})^2 +|Y_{0}^{\pm}|^2 <0.
\end{equation}
Due to the continuity expressions \eqref{eq:Xiagree} and \eqref{eq:alphanus} for $\widetilde\alpha_0^\pm$, and the fact that $|\widetilde{\nabla}v_0|^2=-1$, the flow of $\widetilde{\nabla}v_0$ splits the spacetime homeomorphically as $\widetilde{N}^{n+1}=\mathbb{R}\times \mathcal{M}^n$. Since the $\widetilde\nabla v_i$ are tangent to $\mathcal{M}^n$ away from $\widehat{\Sigma}^n\cap\mathcal{M}^n$ for $i=1,\dots,n$ and satisfy 
\begin{equation}
\langle\widetilde{\nabla}v_i,\widetilde{\nabla}v_j\rangle =\delta_{ij},
\end{equation}
the restriction $\mathcal{V}|_{\mathcal{M}^n}$ is a local isometry into a constant time slice of Minkowski space away from $\widehat{\Sigma}^n$, and is a local homeomorphism across $\widehat\Sigma^n$. This in particular implies that $\mathcal{M}^n$ is a complete metric space. The local homeomorphism property combined with completeness can be used to show that $\mathcal{V}|_{\mathcal{M}^n}$ is a covering map. Since its target is simply connected, we find that $\mathcal{M}^n$ is homeomorphic to $\mathbb{R}^n$, and hence $\widetilde{N}^{n+1}$ is Lipschitz homeomorphic to $\mathbb{R}^{1,n}$ via the map $\mathcal{V}$. Since this manifold has a single end it follows that $\widetilde{N}^{n+1}=N^{n+1}$, and so $N^{n+1}$ is Lipschitz homeomorphic to $\mathbb{R}^{1,n}$.

This homeomorphism $N^{n+1}\to\mathbb{R}^{1,n}$ may fail to be smooth with respect to a generic smooth structure on $N^{n+1}$. However, if we consider the differentiable structure in which the map $\mathcal{V}$ becomes a chart, then this map is tautologically smooth and thus yields a diffeomorphism between $N^{n+1}$ and $\mathbb{R}^{1,n}$. Furthermore, since $\langle\widetilde{\nabla}v_a,\widetilde{\nabla}v_b\rangle=\eta_{ab}$ holds throughout $N^{n+1}$ where $\eta$ is the canonical expression for the Minkowski metric, we find that $N^{n+1}$ is isometric to $\mathbb{R}^{1,n}$.
\end{proof}

\bibliographystyle{plain}

\bibliography{refs}

\end{document}